\documentclass[a4paper,leqno,11pt]{amsart}

\usepackage{amsfonts,amssymb,verbatim,amsmath,amsthm,latexsym,textcomp,amscd}
\usepackage{latexsym,amsfonts,amssymb,epsfig,verbatim}
\usepackage{amsmath,amsthm,amssymb,latexsym,graphics,textcomp}
\usepackage{mathtools}
\usepackage{paralist}
\usepackage{graphicx}
\usepackage[usenames,dvipsnames,x11names]{xcolor}
\usepackage{url}
\usepackage{enumerate}
\usepackage[mathscr]{euscript}
\usepackage{tikz}
\usepackage{dsfont}
\usetikzlibrary{matrix}

\input xy

\xyoption{all}

\theoremstyle{plain}
\newtheorem*{thma}{Theorem A}
\newtheorem*{thmb}{Theorem B}
\newtheorem*{thmc}{Theorem C}

\theoremstyle{definition}
\newtheorem{theorem}{Theorem}[section]
\newtheorem{prop}[theorem]{Proposition}
\newtheorem{lemma}[theorem]{Lemma}

\newtheorem{defn}[theorem]{Definition}

\newtheorem{rmk}[theorem]{Remark}

\newtheorem{exam}[theorem]{Example}

\newtheorem{subsec}[theorem]{}
\newtheorem{thm}[theorem]{Theorem}
\newtheorem{notation}[theorem]{Notation}

\theoremstyle{remark}
\swapnumbers
\newtheorem{ack}[theorem]{Acknowledgements}

\theoremstyle{definition}

\newcommand{\C}{{\mathbb C}}

\newcommand{\Hyp}{{\mathbb H}}

\newcommand{\Z}{{\mathbb{Z}}}
\newcommand{\R}{{\mathbb R}}
\newcommand{\Q}{{\mathbb Q}}

\newcommand{\Hom}{\mbox{Hom}}

\newcommand{\upi}{\underline{\pi}}

\newcommand{\uE}{\underline{E}}
\newcommand{\uZp}{\underline{\Z/p}}

\newcommand{\OG}{\mathcal{O}_G}

\newcommand\A{{\mathbb A}}

\newcommand\BB{{\mathcal B}}
\newcommand\CC{{\mathcal C}}

\newcommand\FF{{\mathcal F}}

\newcommand\LL{{\mathcal L}}
\newcommand\MM{{\mathcal M}}

\newcommand\PP{{\mathcal P}}

\newcommand\RR{{\mathcal R}}

\newcommand\TT{{\mathcal T}}
\newcommand\UU{{\mathcal U}}

\newcommand\PMF{{\PP\kern-2pt\MM\FF}}

\newcommand\PML{{\PP\kern-2pt\MM\LL}}

\newcommand\ep{\epsilon}

\newcommand\tr{\operatorname{tr}}

\newcommand{\pair}[1]{\langle #1\rangle}
\newcommand{\Id}{\mbox{Id}}
\newcommand{\Ext}{\mbox{Ext}}

\newcommand{\T}{{\mathbb T}}

\newcommand{\fsubd}{\mathrel{{\scriptstyle\searrow}\kern-1ex^d\kern0.5ex}}
\newcommand{\bsubd}{\mathrel{{\scriptstyle\swarrow}\kern-1.6ex^d\kern0.8ex}}
\newcommand{\fsubeq}{\mathrel{\raise-.7ex\hbox{$\overset{\searrow}{=}$}}}
\newcommand{\bsubeq}{\mathrel{\raise-.7ex\hbox{$\overset{\swarrow}{=}$}}}

\newcommand{\tsh}[1]{\left\{\kern-.9ex\left\{#1\right\}\kern-.9ex\right\}}

\newcommand{\ra}{\rightarrow}

\newenvironment{myeq}[1][]
{\stepcounter{theorem}\begin{equation}\tag{\thetheorem}{#1}}
	{\end{equation}}

\newenvironment{mysubsection}[2][]
{\begin{subsec}\begin{upshape}\begin{bfseries}{#2.}
			\end{bfseries}{#1}}
		{\end{upshape}\end{subsec}}
\usepackage{tikz-cd}

\newcommand{\cA}{\mathcal A}

\newcommand{\res}{\mathit{res}}

\newcommand{\uA}{\underline{A}}

\newcommand{\uM}{\underline{M}}

\tikzcdset{scale cd/.style={every label/.append style={scale=#1},
		cells={nodes={scale=#1}}}}

\usepackage{ushort}
\newcommand{\uH}{\ushort{H}}
\newcommand{\uZ}{\underline{\mathbb{Z}}}

\def\Z{\mathbb{Z}}
\def\bp{\begin{prop}}
	\def\ep{\end{prop}}
\def\a{\alpha}
\def\lam{\lambda}
\def\sU{\mathcal{U}}
\def\C{\mathbb{C}}
\def\b*{\bigstar}
\def\ra{\rightarrow}
\def\xra{\xrightarrow}
\numberwithin{equation}{section}
\tikzset{notestyleraw/.append style={align=justify}}
\newcommand{\smas}{\wedge}
\newcommand{\bigwedg}{\bigvee}

\def\tH{\tilde{H}}
\usepackage{microtype} 
\definecolor{frenchblue}{rgb}{0.0, 0.45, 0.73}
\usepackage[colorlinks,citecolor=frenchblue, linkcolor=frenchblue]{hyperref}
\usepackage{soul}
\usepackage[normalem]{ulem}
\usepackage{etoolbox}
\makeatletter
\newcommand{\changeoperator}[1]{%
	\csletcs{#1@saved}{#1@}%
	\csdef{#1@}{\changed@operator{#1}}%
}
\newcommand{\changed@operator}[1]{%
	\mathop{%
		\mathchoice{\textstyle\csuse{#1@saved}}
		{\csuse{#1@saved}}
		{\csuse{#1@saved}}
		{\csuse{#1@saved}}%
	}%
}
\makeatother
\changeoperator{prod}
\changeoperator{coprod}
\changeoperator{bigotimes}

\newcommand{\Zp}{\underline{{\mathbb Z}/p}}
\newcommand{\Zt}{\underline{{\mathbb Z}/2}}
\usepackage{mathtools}
\newcommand{\bH}{{\mathbb H}}
\usepackage[]{mdframed}
\newcommand{\pt}{\textup{pt}}
\renewcommand{\res}{\operatorname{res}}
\usepackage[disable]{todonotes}
\usepackage{bbm}
\makeatletter
\@namedef{subjclassname@2020}{%
  \textup{2020} Mathematics Subject Classification}
\makeatother

\title[]{Equivariant cohomology of projective spaces}
\author{Samik Basu, Pinka Dey, Aparajita Karmakar}

\address{Stat-Math Unit,
	Indian Statistical Institute,
	B. T. Road, Kolkata-700108, India.}
\email{samik.basu2@gmail.com; samikbasu@isical.ac.in}

\address{Stat-Math Unit,
	Indian Statistical Institute,
	B. T. Road, Kolkata-700108, India.}
\email{pinkadey11@gmail.com;}

\address{Stat-Math Unit,
	Indian Statistical Institute,
	B. T. Road, Kolkata-700108, India.}
\email{aparajitakarmakar@gmail.com;}


\begin{document}
	\begin{abstract}
We compute the equivariant homology and cohomology of projective spaces with integer coefficients. More precisely, in the case of cyclic groups, we show that the cellular filtration of the projective space $P(k\rho)$, of lines inside copies of the regular representation, yields a splitting of $H\uZ\wedge P(k\rho)_+$ as a wedge of suspensions of $H\uZ$. This is carried out both in the complex case, and also in the quaternionic case, and further, for the $C_2$ action on $\C P^n$ by complex conjugation. We also observe that these decompositions imply a degeneration of the slice tower in these cases. Finally, we describe the cohomology of the projective spaces when $|G|=p^m$ of prime power order, with explicit formulas for $\uZp$-coefficients. Letting $k=\infty$, this also describes the equivariant homology and cohomology of the classifying spaces of $S^1$ and $S^3$. 
	\end{abstract}
	\maketitle

	\section{Introduction}
	The purpose of this paper is to discuss new calculations for the equivariant cohomology of complex projective spaces. Given a complex representation $V$ of a group $G$, one obtains a ``linear'' $G$-action on $P(V)$ $=$ the space of lines in $V$. The underlying space here is $\C P^{\dim(V)-1}$ whose homology computation is well-known.  The Borel-equivariant cohomology, which is the cohomology of the Borel construction, is easy to calculate as the space $P(V)$ has non-empty fixed points. 
	
	The equivariant cohomology used in this paper refers to an ``ordinary'' cohomology theory represented in the equivariant stable homotopy category \cite[Ch. XIII \S 4]{May-EquivBook}.  In this context, the word ``ordinary'' means that the equivariant homotopy groups of the representing spectrum $HM$  are concentrated in degree $0$. At this point, one notes that the equivariant stable homotopy category possesses additional structure which makes the equivariant homotopy groups part of a Mackey functor \cite[Ch. IX \S 4]{May-EquivBook}. Conversely, given a Mackey functor $\uM$, one has an associated ``ordinary'' equivariant cohomology  theory represented by an Eilenberg-MacLane spectrum $H\uM$ \cite[Theorem 5.3]{GreenleesMay}, which is unique up to homotopy.
	
	A Mackey functor $\uM$ \cite{Dre72} comprises a pair of functors $(\uM_\ast, \uM^{\ast})$ from the category $\OG$ (of finite $G$-sets) to Abelian groups taking disjoint unions to direct sums, such that $\uM_\ast$ is covariant and $\uM^\ast$ is contravariant, taking the same value on a given $G$-set $S$, denoted $\uM(S)$. These are required to be compatible in the sense of a double coset formula \cite{TW95}. The covariant  structure gives restriction maps 
	\[ \res^H_K : \uM(H) \to \uM(K), \mbox{ for } K \subset H,\]      
	and the contravariant structure gives transfer maps 
	\[ \tr^H_K : \uM(K) \to \uM(H), \mbox{ for } K\subset H.\]
The two important examples in the context of equivariant cohomology are the Burnside ring Mackey functor $\uA$, and the constant Mackey functor $\uZ$. The  Mackey functor $\uZ$ sends each $G/H \mapsto \Z$, with $\res^H_K=\Id$ and $\tr^H_K=$ multiplication by the index $[H: K]$. The Burnside ring Mackey functor $\uA$ sends $G/H\mapsto A(H)$, the ring generated by isomorphism classes of finite $H$-sets. The restriction maps $\res^H_K$ for $\uA$ are described as the restriction of  the action of $H$ to $K$, and the transfer maps $\tr^H_K$ are described by induction $S \mapsto H\times_K S$.   
	
	The category of Mackey functors is an Abelian category. It also has a symmetric monoidal structure given by $\Box$, whose unit object is the Burnside ring Mackey functor $\uA$. The constant Mackey functor $\uZ$ is  a commutative monoid, which implies that the cohomology groups with $\uZ$-coefficients possess a graded commutative ring structure. The spectrum $H\uA$ is a homotopy commutative ring spectrum, and the commutative monoids give homotopy commutative $H\uA$-algebras. However, if one tries to rigidify the construction of Eilenberg-MacLane spectra into a functor taking values in equivariant orthogonal spectra, there are obstructions coming from norm maps \cite{Ull13,HHR}. 
	
	Our focus in this paper is on computations of equivariant cohomology for $G$-spaces. A simple-minded approach would be to break up the spaces into equivariant cells, and compute via cellular homology.  An equivariant $G$-CW complex has cells of the form $G/H\times D^n$, which are attached along maps from $G/H\times S^{n-1}$ onto lower skeleta. Via this argument, one shows $H^n_G(X;\uZ)\cong H^n(X/G;\Z)$. While working through concrete examples like the projective spaces $P(V)$, or for $G$-manifolds, we see that there is no systematic way of breaking these up into cells of the form $G/H\times D^n$ or identifying the space $X/G$. In these cases, the spaces may be naturally built out of cells of the form $G\times_H D(V)$\cite{Was69}, where $V$ is a unitary $H$-representation, and $D(V)$ stands for the unit disk 
	\[ D(V)= \{ v\in V \mid \pair{v,v} \leq 1 \} .\] 

	In the equivariant stable homotopy category, the representation spheres $S^V$ (defined as the one-point compactification of $V$) are invertible in the sense that there is a $S^{-V}$, such that $S^{-V} \wedge S^V \simeq S^0$. Consequently, the equivariant cohomology becomes $RO(G)$-graded \cite[Ch. XIII]{May-EquivBook}, and so, computations for the $G$-CW complexes above require the knowledge of $H^\alpha_G(G/H;\uM)$ for $\alpha \in RO(G)$ and $H\subset G$. Such computations exist in the literature only for a handful of finite groups, namely, for $G=C_p$ for $p$ prime \cite{Lewis}, $G=C_{p_1\cdots p_k}$ for distinct primes $p_i$ \cite{BG19,BG20}, and $G=C_{p^2}$ \cite{Zeng}. For restricted $\alpha$ belonging to certain sectors of $RO(G)$ more computations are known \cite{HHR, HHR-slicess,HoKr17, BasuDey}.   
	
	Equivariant cohomology of $G$-spaces have been carried out for many of the groups above \cite{Haz21, Hog21, BasuDeyKarmakar-NYJM}. Most of these computations are done in the case where the complete calculations for  $H^\alpha_G(G/H;\uZ)$ are known. Many of these occur in cases where the cohomology is a free module over $\pi_\bigstar H\uZ$. There are various structural results which imply the conclusion that the cohomology is a free module \cite{May20,BasuGhosh-cp}. For the equivariant projective spaces, this is  particularly relevant, and although, the entire knowledge of $\pi_\bigstar^{C_n}H\uZ$ is still unknown, we are able to prove that the cohomology $P(V)$ is free when $V$ is a direct sum of copies of the regular representation. (see Theorems \ref{thm main simple} and \ref{quatadd})
	\begin{thma} Let $G=C_n$. We have the following decompositions. \\
a) Write $\phi_{0}= 0$ and $\phi_{i} = \lam^{-i}(1_\C+\lam + \lam^2 +\cdots + \lam^{i-1})$ for $i> 0$. Then, 		
$$H\uZ\smas P(m\rho_\C)_+ \simeq \bigvee_{i=0}^{nm-1}H\uZ\smas S^{\phi_{i}},$$ 
		$$H\uZ\smas B_GS^1_+ \simeq \bigvee_{i=0}^{\infty}H\uZ\smas S^{\phi_{i}}.$$
b) For the quaternionic case, we have for $W_k = \lambda^{-k}\Big(\sum_{i=0}^{k-1}(\lambda^i + \lambda^{-i})\Big)$,
$$H\uZ \smas P_{\mathbb H}(m\rho_\Hyp)_+ \simeq \bigvee_{i=0}^{mn-1}H\uZ\smas S^{W_i},$$ 
		$$H\uZ\smas B_GS^3_+ \simeq \bigvee_{i=0}^{\infty}H\uZ\smas S^{W_{i}}.$$
	\end{thma} 
In the above expression, $\lambda$ refers to the one-dimensional complex $C_n$-representation which sends a fixed generator $g$ to  $e^{\frac{2\pi i}{n}}$, and it's powers are taken with respect to the complex tensor product. The second implications above come from the identification $B_GS^1 \simeq P(\UU) \simeq \varinjlim_m P(m\rho)$, and $B_GS^3 \simeq P_\Hyp(\Hyp\otimes_\C \UU) \simeq \varinjlim_m P_\Hyp(m\rho_\Hyp)$, where $\UU$ is a complete $G$-universe. We also carry out the computation for general $V$ when the group $G$ equals $C_p$ (Theorem \ref{thm V with uneven coeff}). One may view this as a simplification of the results in \cite{Lewis} in the case of the Mackey functor $\uZ$. For the group $C_2$, we consider $\C P^n_\tau$ where $C_2$ acts on $\C P^n$ by complex conjugation, and compute it's equivariant homology (Theorem \ref{prop:cpnconj}). In this context, one should note that results such as the above theorem are not expected for $\uA$-coefficients once the group contains either $C_{p^2}$ or $C_p\times C_p$ \cite[Remark 2.2]{Lew92},\cite{FL04}. 

As an application for the homology decomposition in Theorem A, we reprove a theorem of Caruso \cite{Caruso} stating that the cohomology operations expressible as a product of the even degree Steenrod squares  over $\Z/2$ do not occur as restriction of integer degree $C_2$-equivariant cohomology  operations. If we allow the more general $RO(C_2)$-graded operations, there are those that restrict to $Sq^i$ for every $i$ \cite{Voe03}. For the group $C_p$, the same result holds for the products of the Steenrod powers $P^i$. 

A careful analysis of the cellular filtration of the projective spaces $P(V)$ for $V=m\rho$ shows that the induced filtration on $\Sigma^2 H\uZ \wedge P(V)$ matches the slice filtration. The slice tower was defined  as the equivariant analogue of the Postnikov tower using the localizing subcategory generated by $\{{G/H}_+\wedge S^{k\rho_H} \mid k|H|\geq n\}$ instead of the spheres of the form $\{G/H_+\wedge S^n\}$. The slice filtration played a critical role in the proof of the Kervaire invariant one problem \cite{HHR} and has been widely studied since. Usually, the slice tower is an involved computation even for the spectra $\Sigma^n H\uZ$. However, for the complex projective spaces and the quaternionic projective spaces, we discover that the slice tower for the $\uZ$-homology becomes amazingly simple. More precisely, we prove the following theorem in this regard. (see Theorems \ref{thm:slice} and \ref{thm:slicequa})
\begin{thmb} Let $G=C_n$. \\
a) The slice towers of $ \Sigma^2 P(\mathcal{U})_+\smas H\uZ $ and $ \Sigma^2 P(m\rho)_+\smas H\uZ $ are degenerate and these spectra are a wedge of slices of the form $ S^V\smas H\uZ $. \\
b) The slice towers of $ \Sigma^4 P(\mathcal{U}_\bH)_+\smas H\uZ $ and $ \Sigma^4 P(m\rho_\bH)_+\smas H\uZ $ are degenerate and these spectra are a wedge of slices of the form $ S^V\smas H\uZ $.
\end{thmb}

We proceed to describe the cohomology ring structure for the projective spaces. The basic approach towards the computation comes from \cite{Lewis}, which deals with the case $G=C_p$. In this case, we observe that the generators are easy to describe, but the relations become complicated once the order of the group increases. For $G=C_n$, we show that $H^\bigstar_G(P(V);\uZ)$ are multiplicatively generated by classes $\alpha_{\phi_d}$ for $d\mid n$, in degree $\sum_{i=-d}^{-1} \lam^i$ (Proposition \ref{prop:multi gene}). However, the relations turn out to be difficult to write down in general, so we restrict our attention to prime powers $n$. In the process of figuring out the generators, we realize that there are exactly $m$ relations $\rho_j$ for $1\leq j\leq m$ ($n=p^m$), of the form 
\[u_{\lam^{p^{j-1}}-\lam^{p^j}}\alpha_{\phi_{p^j}} = \alpha_{\phi_{p^{j-1}}}^p + \mbox{ lower order terms}.\]
The explicit form turns out to be quite involved with $\uZ$-coefficients, so we determine them modulo $p$, and prove the following results. (see Theorems \ref{ringstrs1}, \ref{ringstr2}, and  \ref{cptau2}) 
\begin{thmc} 
a) 	$H^\bigstar_{C_2}(\C P^\infty_\tau;\uZ)\cong H^\bigstar_{C_2}(\pt)[\epsilon_{1+\sigma}].$\\
b) The cohomology ring 
	\[ H^\bigstar_G (B_GS^1;\uZp) \cong H^\bigstar_G(\pt;\uZp)[\a_{\phi_0},\cdots,\a_{\phi_m}]/(\rho_1,\cdots, \rho_m).\]
	The relations $\rho_r$ are described by 
	\[\rho_r= u_{\lam^{p^{r-1}}-\lam^{p^{r}}}\a_{\phi_{p^{r}}}- \TT_{r-1}^p + a_{\lambda^{p^{r-1}}}^{p-1}\TT_{r-1}\big(\prod_{i=0}^{r-2} \cA_i\big)^{p-1}, \]
	where $\TT_j$ and $\cA_j$ are defined in \eqref{btnot}. \\
c) The cohomology ring 
\[ H^\bigstar_G (B_GS^3;\uZp) \cong H^\bigstar_G(\pt;\uZp)[\beta_{2\phi_0},\cdots,\beta_{2\phi_m}]/(\mu_1,\cdots, \mu_m).\]
	The relations $\mu_r$ are described by 
	\[\mu_r= (u_{\lam^{p^{r-1}}-\lam^{p^{r}}})^2\beta_{2\phi_{p^{r}}}- \LL_{r-1}^p + a_{\lambda^{p^{r-1}}}^{2(p-1)}\LL_{r-1}\big(\prod_{i=0}^{r-2} \CC_i\big)^{p-1}, \]
	where $\LL_j$ and $\CC_i$ are defined in Theorem \ref{ringstr2}. 
\end{thmc}
	
	\begin{mysubsection}{Organization}
In \S \ref{eqic}, we recall results in equivariant homotopy theory that are useful from the viewpoint of ordinary cohomology. In \S \ref{eqicyc}, we recall previously known computations for $\uZ$-coefficients, and extend them as necessary for the following sections. We prove the homology decompositions for projective spaces in \S \ref{sec:additive}. These are applied to cohomology operations in \S \ref{sec cohop}, and the slice tower in \S \ref{slice}. The ring structures are described in \S \ref{cohring}.
	\end{mysubsection}
	
	\begin{notation}  We use the following notations throughout the paper.
		\begin{itemize} 
\item Throughout this paper, $G$ denotes the cyclic group of order $n$, and $g$ denotes a fixed generator of $G$.  The unit sphere of an orthogonal $G$-representation $V$  is denoted by $ S(V) $, the unit disk by $D(V)$, and $S^V$ the one-point compactification $\cong D(V)/S(V)$.  
\item  We write $ 1_\C $ for the trivial complex representation and $ 1 $ for the real trivial representation, and the regular representation $ \rho $ to mean $ \rho_\C $ (the complex regular representation) if not specified. 	The irreducible complex representations of $G$ are $1$-dimensional, and up to isomorphism are listed as $1_\C,\lambda, \lambda^2,\dots, \lambda^{n-1}$ where $\lambda$ sends $g$ to  $e^{2\pi i/n}$, the $n^{th}$ root of unity. 
\item In the case $n$ is even, we denote the sign representation by $\sigma$. This is a $1$-dimensional real representation.  
\item  The non-trivial real irreducible representations other than $\sigma$ are the underlying real representations of the complex irreducible representations. The realization of $\lambda^i$ is also denoted by the same notation. Note that $\lambda^i$ and $\lambda^{n-i}$ are conjugate and hence their realizations are isomorphic by the natural $\R$-linear map $z\mapsto \bar{z}$ which reverses orientation.
\item  Unless specified, the cohomology groups are taken  with $ \uZ $-coefficients and suppressed from the notation.
\item  The notation  $\mathcal{U}$ is used for the complete  $G$-universe.
			\item The notation $ p $ is used for a prime not necessarily odd. $ \mathcal{A}_p $ denotes the non-equivariant $\pmod p $ Steenrod algebra.
			\item The linear combinations $ \ell -(\Sigma_{d_i \mid n}~~b_i\mspace{2mu} \lambda^{d_i}) $  with $\ell  \in \Z$ and $ b_i\in \Z_\ge0 $  is denoted by $ \bigstar^e \subset RO(G) $. In the case $n$ is even,  we denote by $\bigstar_{\textup{div}}$ the gradings of the form $\ell -(\Sigma_{d_i \mid n}~~b_i\mspace{2mu} \lambda^{d_i}) - \epsilon \sigma$ where $\ell \in \Z$, $b_i \in \Z_{\geq 0}$ and $\epsilon \in \{0,1\}$. 
			\item $ \C P^\infty_\tau $ is the complex projective space with the $ C_2 $-action given by conjugation.
		\end{itemize}
		
	\end{notation}
	
	\begin{ack}
The first author would like to thank Surojit Ghosh for some helpful conversations in the context of computations of $\pi_\bigstar H\uZ$. The research of the second author was supported by the NBHM grant no. 16(21)/2020/11. The results of this paper are a part of the doctoral research work of the third author.
\end{ack}

	\section{Equivariant cohomology with integer coefficients}\label{eqic}
	We recall some features of the equivariant cohomology with coefficients in a Mackey functor, with a particular emphasis on $ \uZ $-coefficients, restricting our attention to cyclic groups $G  $.  For such 
	 $G$, a $G$-Mackey functor $\underline{M}$ consists of Abelian  groups $ \underline{M}(G/H) $ with $ G/H $-action, for every subgroup $H\leq G$, and they are related via the following maps.
	\begin{enumerate}
		\item The \textit{transfer} map $\tr^H_K\colon  \uM(G/K) \to \uM(G/H)$
		\item The \textit{restriction} map $\res^H_K \colon\uM(G/H) \to \uM(G/K)$
	\end{enumerate}
	 for $K\le H\le G$. The composite $ \res^H_L \tr^H_K $ satisfies a double coset formula (see \cite[\S 3]{HHR}). 
	 \begin{exam}
	 	The Burnside ring Mackey functor, $\uA$ is  defined by $\uA(G/H)=A(H)$. Here $ A(H) $ is the Burnside ring of $H$, i.e., the group completion of the monoid of finite $H$-sets up to isomorphism. The transfer maps are defined by inducing up the action : $S \mapsto H\times_K S$ for $K\leq H$, and the restriction maps are given by  restricting the action. For the $ K $-set $ K/K $, the double coset formula takes the form $  \res^H_L \tr^H_K(K/K)=\res_L^H(H/K)=\textup{union of double cosets $ L \,\backslash H / K $}. $
	 \end{exam}
	\begin{exam}
		The constant $G$-Mackey functors are defined as follows. For an Abelian group $ C $,  the constant $G$-Mackey functor $\underline{C}$ is defined as
		\[
		\underline{C}(G/H)=C, \res^H_K=\Id, \tr^H_K=[H:K].
		\]
		for $ K\le H\le G $. The double coset formula is given by  $  \res^H_L \tr^H_K(x)=[H:K]x$, for an element $x\in C  $. We may also define its dual Mackey functor  $\underline{C}^*$ by
		\[
		\underline{C}^*(G/H)=C, \res^H_K=[H:K], \tr^H_K=\Id.
		\]
	\begin{exam}
		For the group $ C_p $,  the  Mackey functor $ \langle  \Z/p  \rangle  $  is defined  by 
		\[
		\langle   \Z/p  \rangle  (C_p/C_p)= \Z/p,~~\langle   \Z/p  \rangle  (C_p/e)= 0,~~ \res^{C_p}_e=0,~~ \tr^{C_p}_e = 0.
		\]
	\end{exam}
	The following  Mackey functor will appear in \S \ref{sec cohop}.
		  \begin{exam}
		  	For the group $ C_2 $, we have the  Mackey functor $ \langle  \Lambda  \rangle  $  described by 
		\[
		\langle   \Lambda  \rangle  (C_2/e)= \Z/2,~~\langle   \Lambda  \rangle  (C_2/C_2)= 0,~~ \res^{C_2}_e=0,~~ \tr^{C_2}_e = 0.
		\]
		  \end{exam}
	\end{exam}
The equivariant stable homotopy  category is the homotopy  category of equivariant orthogonal spectra  \cite{MandellMay}. The Eilenberg-MacLane spectra are those whose integer graded homotopy  groups vanish except in degree 0. The following describes its relation with the Mackey functors. 
	\begin{thm}\cite[Theorem 5.3]{GreenleesMay}
		For every Mackey functor $\uM$, there is an Eilenberg-MacLane $G$-spectrum $H\uM$ which is unique up to isomorphism in the equivariant stable homotopy category. 
	\end{thm}

	In the equivariant stable homotopy category the objects $ S^{V} $ are invertible for a representation  $ V$. For a $G$-spectrum $X$, the equivariant homotopy groups have the structure of a Mackey functor $\upi^G_n(X)$, which on objects assigns the value
	\[ \upi^G_n(X)(G/H):= \pi_n(X^H).
	\]
	The grading may be extended to  $ \alpha\in  RO(G)$, the real representation ring of $ G $,    as
	\[
	\upi_\alpha^G(X)(G/K) \cong \mbox{ Ho-}G\mbox{-spectra } (S^\alpha \smas G/K_+ , X ),
	\]
	which is  isomorphic to $\pi_\alpha^{K}(X)$. Analogously, equivariant homology and cohomology theories are $RO(G)$-graded and Mackey functor valued, which on objects is defined by 
	\[
	\uH^\alpha_G(X;\uM)(G/K) \cong \mbox{ Ho-}G\mbox{-spectra } (X\smas G/K_+ , \Sigma^\alpha H\uM),
	\]
	\[
	\uH_\alpha^G(X;\uM)(G/K) \cong \mbox{ Ho-}G\mbox{-spectra } (S^\alpha \smas G/K_+ , X \smas H\uM).
	\]
	  For a constant  Mackey functor $ \underline{C} $,   the integer graded groups at orbit  $G/G$ compute the cohomology of  the orbit space of $X$ under the $G$-action, i.e., $H^{n}_{G}(X;\underline{C}) = H^{n}(X/G;C)$.

As $ H\uZ $ is a ring spectrum, the Mackey functor $ \uZ $ has a multiplicative structure, i.e, it is a commutative Green functor \cite[chapter XIII.5]{May-EquivBook}. As a consequence, $ \uH^\alpha_G(X;\uZ) $ are $ \uZ $-modules. These modules satisfy the following property.
\begin{prop}\cite[Theorem 4.3]{Yos83}\label{trres}
	For any $ \uZ $-module G-Mackey functor $ \uM $, $ \tr^H_K\res^H_K $ is multiplication by the  index $ [H:K] $ for $ K\le H\le G $.
\end{prop}
For an element $ \alpha\in RO(G) $ such that $ |\alpha^H| =0$ for all $ H\le G $, the representation sphere $ S^\alpha $ belongs to the Picard group of $ G $-spectra.   We note from \cite[Theorem B]{Ang21} that for such $\alpha$
\begin{myeq}\label{eq un6}
	\upi_0^G(H\uZ\smas S^\alpha)\cong  \uZ.
\end{myeq}
We recall a few important classes which generate a portion of the  ring $ \pi_\bigstar^G(H\uZ) $.
	\begin{defn}\cite[\S3]{HHR}
		For a $G$-representation $V$, consider the inclusion $S^0\hookrightarrow S^V$.  Composing with the unit map $S^0 \rightarrow H\uZ$, we obtain $S^0 \hookrightarrow S^V\ra S^V\smas H\uZ$ which represents an element in $\pi^{G}_{-V}(H\uZ)\cong \tH^{V}_{G}(S^0;\uZ)$. This class is denoted by $a_V$.
	\end{defn}
	
	\begin{defn}\cite[\S3]{HHR}
		For an oriented $G$-representation of dimension $n$,  $\underline{\pi}^{G}_{n-V}(H\uZ)= \uZ$ \cite[Example 3.10]{HHR}. Define $u_V\in \pi^{G}_{n-V}(H\uZ)\cong \tH^{G}_{n}(S^{V};\uZ)\cong \Z$ to be the generator that restricts to the choice of orientation in $\uH^{G}_{n}(S^V;\uZ)(G/e)\cong \tH_{n}(S^n; \Z)\cong \Z$.
	\end{defn}
	The following result simplifies calculations for $ \pi_\bigstar^G(H\uZ) $. 
	\begin{prop}\cite[Theorem 3.6]{BasuDeyKarmakar-NYJM},  \label{prop:smash-rep}	
		If $ \gcd(s,n)=1 $, then 
		\[
		H\uZ\smas S^{\lambda^i} \simeq H\uZ\smas S^{\lambda^{s i}}.
		\]
	\end{prop}	
	The Proposition \ref{prop:smash-rep} implies  $\Sigma^{\lambda^i-\lambda^{s i}} H\uZ \simeq H\uZ$ if $ \gcd(s,n)=1 $. That is, in the graded commutative ring $\pi^G_\bigstar H\uZ$, there are invertible elements  $ u_{\lambda^{si}-\lambda^i}$ in degrees $\lambda^{i} - \lambda^{si}$ whenever $\gcd (s,n)=1$. Therefore to determine the ring $\pi^G_\bigstar H\uZ$  it is enough to consider the gradings which are linear combinations of $\lambda^d$ for $d\mid n$.\par
	 The linear combinations $ \ell -(\Sigma_{d_i\mid n} ~~b_i\mspace{2mu} \lambda^{d_i})  + \epsilon \sigma\in RO(G)$  with $\ell, b_i\in \Z $ and $\epsilon \in \{ 0, 1\}$  are denoted by $ \bigstar_{\textup{div}} $. The last term $\epsilon \sigma$ occurs only when $|G|$ is even. In the case of $ H\uZ $, the ring $ \pi^G_{\bigstar_{\textup{div}}}(H\uZ) $ is also obtained from $ \pi^G_{\bigstar}(H\uZ) $ by identifying all the
	 $  u_{\lambda^{si}-\lambda^i} $ with $ 1 $. More precisely,
	\[
	\pi^G_{\bigstar_{\textup{div}}}(H\uZ) \cong\pi^G_{\bigstar}(H\uZ)/(u_{\lambda^{si}-\lambda^i}-1) ,\quad \textup{~and~} \quad  \pi^G_{\bigstar}(H\uZ)\cong \pi^G_{\bigstar_{\textup{div}}}(H\uZ)[u_{\lambda^{si}-\lambda^i}\mid \gcd(s,n)=1].
	\]
	
	In fact, we make  a choice of $ u_{\lambda^{si}-\lambda^i} $ such that $ \res_e^G(u_{\lambda^{si}-\lambda^i})=1\in H^0(\pt) $. This implies that $ u_{\lambda^{si}-\lambda^i}\cdot u_{\lambda^i}=u_{\lambda^{si}} $. The following proposition describes the relation between $  a_{\lambda^{si}} $ and $ a_{\lambda^{i}}  $ in $ \pi^G_{\bigstar_{\textup{div}}}(H\uZ) $ in case $ G $ is of prime power order.
	\begin{prop}\label{prop:un5}
	Let $ G=C_{p^m} $.	 If $ \gcd(s,p)=1 $, then 
	\[
	 u_{\lambda^{i}-\lambda^{si}}a_{\lambda^{si}}= sa_{\lambda^i}.
	\]
	Thus, in the ring $ \pi^G_{\bigstar_{\textup{div}}}(H\uZ) $,
		 \[
		 a_{\lambda^{si}}= sa_{\lambda^i}.
		 \]
	\end{prop} 
\begin{proof}
	By	Proposition \ref{prop:smash-rep} we have 
\begin{myeq}\label{un5}
	H\uZ\smas S^{\lambda^i} \simeq H\uZ\smas S^{\lambda^{s i}}.
\end{myeq}
	Hence there exists $ u_{\lambda^{si}-\lambda^i}\in \pi_0^G(H\uZ\smas S^{\lambda^{si}-\lambda^i}) $ such that the  composition
\[
\psi: H\uZ\smas S^{\lambda^i}\xrightarrow{id\smas u_{\lambda^{si}-\lambda^i}} H\uZ\smas S^{\lambda^i}\smas H\uZ\smas S^{\lambda^{si}-\lambda^i}\to H\uZ\smas S^{\lambda^{si}}
\]	
induces the equivalence \eqref{un5}. Since all the fixed point dimensions of $ \lambda^{si}-\lambda^i $ are zero,
\[
\upi_0^G(H\uZ\smas S^{\lambda^{si}-\lambda^i})\cong \uZ
\]
by \eqref{eq un6}. Choose $ u_{\lambda^{si}-\lambda^i}\in \pi_0^G(H\uZ\smas S^{\lambda^{si}-\lambda^i}) $ to be the element such that $ \res^G_e(u_{\lambda^{si}-\lambda^i})=1 $.
Multiplication by $u_{\lambda^{si}-\lambda^i}$
\[
\psi_*: \pi_\alpha^G(H\uZ\smas S^{\lambda^i}) \to \pi_\alpha^G(H\uZ\smas S^{\lambda^{si}}) 
\]
sends $ a_{\lambda^i} $ to $ u_{\lambda^{si}-\lambda^i}\cdot a_{\lambda^{i}} $. 
Consider the map
\[
\phi_{\lambda^{i}-\lambda^{si}}: S^{\lambda^i}\to S^{\lambda^{si}}.
\]
under which $ z \mapsto z^s$, that is, the underlying degree of $ \phi_{\lambda^{i}-\lambda^{si}} $ is $ s $. Since $  \gcd(s,p)=1  $, we may choose $ s' $ such that $ s\cdot s'=1+tp^m $. Then 
$$ u_{\lambda^{si}-\lambda^i}= s'\cdot \phi^{H\uZ}_{\lambda^{i}-\lambda^{si}}-t\cdot \tr_e^G(1)
$$
as $ \res^G_e(u_{\lambda^{si}-\lambda^i})=s\cdot s'-tp^m=1 $, where $ \phi^{H\uZ}_{\lambda^{i}-\lambda^{si}} $ is the Hurewicz image of $ \phi_{\lambda^{i}-\lambda^{si}} $.  This implies
 \[
 u_{\lambda^{si}-\lambda^i}\cdot a_{\lambda^i}=s'\cdot \phi_{\lambda^{i}-\lambda^{si}}\cdot a_{\lambda^i} -t\cdot \tr_e^G(1)\cdot a_{\lambda^i}= s'\cdot a_{\lambda^{si}}.
 \]
 Consequently $ a_{\lambda^{si}}=s\cdot a_{\lambda^i} $.
\end{proof}
In particular, we obtain the following again assuming $|G|=p^m$.
	\begin{prop}\label{prop:rel betw a lambda}
		Let $ d =p^k$ be a divisor of $ p^m $ and $ 1\le i<d $.	 Note that $ i-d $ and $ i $ have the same $ p $-adic valuation. Then 
		\[
		a_{\lambda^{i-d}}=\Theta_{i,d} \cdot a_{\lambda^i},
		\]
		where   $\Theta_{i,d}=\frac{i-d}{i}$ which is well defined in $ \Z/p^m $. 
	\end{prop}
	The following computations of $\upi_\bigstar^{C_p}(H\Zp)$ \cite[Appendix B]{Ferland-thesis}	 will help us in the following sections. For an odd prime  $ p $ we have
	\begin{myeq}\label{cp-comp zp cof}
		\upi_{\alpha}^{C_p}(H \Zp) = \begin{cases} \Zp & \text{if} \; |\alpha| =0, \; |\alpha^{C_p}| \geq 0 \\
			\Zp^* & \text{if} \; |\alpha| =0, \; |\alpha^{C_p}| < 0 \\
			\langle \Z/p \rangle & \text{if}\;  |\alpha| <0, \; |\alpha^{C_p}| \geq 0, \; \text{and} \; |\alpha| \; \text{even} \\ 
			\langle \Z/p \rangle & \text{if} \;|\alpha|>0, \; |\alpha^{C_p}| <-1,\; \text{and} \; |\alpha| \; \text{odd} \\ 
			0 & \text{otherwise,} \end{cases}
	\end{myeq}
and	for the group $ C_2 $, 
	\begin{myeq}\label{cp-comp z2 cof}
		\upi_{\alpha}^{C_2}(H \Zt) = \begin{cases} \Zt & \text{if} \; |\alpha| =0, \; |\alpha^{C_2}| \geq 0 \\
			\Zt^* & \text{if} \; |\alpha| =0, \; |\alpha^{C_2}| < 0 \\
			\langle  \Lambda  \rangle & \text{if} \;|\alpha|=0, \; |\alpha^{C_2}| =-1, \\ 
			\langle \Z/2 \rangle & \text{if}\;  |\alpha| <0, \; |\alpha^{C_2}| \geq 0, \\ 
			\langle \Z/2 \rangle & \text{if} \;|\alpha|>0, \; |\alpha^{C_2}| <-1, \\ 
			0 & \text{otherwise.} \end{cases}
	\end{myeq}
	
	\begin{mysubsection}{Anderson duality}
Let $I_{\Q}$ and $I_{\Q/\Z}$ be the spectra representing the cohomology theories given by $X \mapsto \Hom(\pi_{-\ast}^G(X), \Q)$ and $X \mapsto \Hom(\pi_{-\ast}^G(X), \Q/\Z)$ respectively. The natural map $\Q \to \Q/\Z$ induces the spectrum map $I_{\Q} \to I_{\Q/\Z}$, and the homotopy fibre is denoted by $I_{\Z}$. For a $G$-spectrum $X$, the \emph{Anderson dual} $I_{\Z}X$ of $X$, is the function spectrum $F(X, I_{\Z})$. For $X=H\uZ$, one easily computes $I_{\Z}H\uZ \simeq H\uZ^\ast\simeq \Sigma^{2-\lambda}H\uZ$ \cite{BG20,Zeng} in the case $G$ is a cyclic group.

In general, for $G$-spectra $E$, $X$, and $\alpha \in RO(G),$ there is short exact sequence  
\begin{myeq}\label{end_dual}
0 \to \Ext_L(\uE_{\alpha -1}(X), \Z) \to \underline{I_{\Z}(E)}^{\alpha}(X) \to \Hom_L(\uE_{\alpha}(X), \Z)\to 0.
\end{myeq}
 In \eqref{end_dual}, $\Ext_L$ and $\Hom_L$ refer to level-wise $\Ext$ and $\Hom$, which turn out to be Mackey functors. In particular, for $E=H\uZ$ and $X=S^0,$ we have the equivalence $\uE_\alpha(X) \cong \upi^G_\alpha (S^0; \uZ)$. Therefore, one may rewrite \eqref{end_dual} as
 \begin{myeq}\label{and_comp}
 0 \to  \Ext_L(\upi_{\alpha+\lambda -3}^{G}(H\uZ), \Z) \to \upi_{-\alpha}^{G}(H\uZ) \to \Hom_L(\upi_{\alpha+\lambda -2}^{G}(H\uZ), \Z) \to 0
 \end{myeq}
 for each $\alpha \in RO(G).$ 
\end{mysubsection}

	\section{$ \pi^G_{\bigstar}(H\uZ) $ for cyclic  groups}\label{eqicyc}
	This section describes various structural results of $ \pi^G_{\bigstar}(H\uZ) $  which helps us to construct the homology decompositions in the later sections.  With Burnside ring coefficients, Lewis \cite{Lewis} first described $ \pi^{C_p}_\bigstar(H\uA) $. The portion of the $  RO(C_{p^n} )$-graded homotopy of $H\uZ$ in dimensions of the form $ k - V $ was described in \cite{HHR-realKthy,HHR-slicess}. Using the Tate square,  $ \pi^{C_2}_\bigstar(H\uZ) $ was determined in \cite{Greenlees} and $ \pi^{C_p}_\bigstar(H\Zp) $ in \cite{BasuGhosh-cp}. For groups of square free order $ \pi^G_\bigstar(H\uZ) $ was explored in \cite{BG20}, and for the group $ C_{p^2} $, $ \pi^G_\bigstar(H\uZ) $ appeared in \cite{Zeng}. 

We use the notation $\pi_{\bigstar^e}^G(H\uZ)$ to denote the part of $\pi_{\bigstar_{\textup{div}}}^G(H\uZ)$ in gradings of the form $k-V$ where $V$ does not contain the sign representation $\sigma$. That is, $\bigstar^e \subset RO(G)$ consists of $l-\sum_{d \mid n} b_d \lambda^d$ with $b_d\geq 0$. 
	\begin{thm}[\cite{BasuDey}] \label{Zhtpy}
		The subalgebra  $ \pi_{\bigstar^e}^G(H\uZ) $ of $ \pi_\bigstar^G(H\uZ) $ is generated over $ \Z $ by the classes $ a_{\lambda^d}, u_{\lambda^d} $ where $ d $ is a divisor of $ n $, $d\ne n$ with relations   
		\begin{myeq}\label{eq:a rel Z}
			\frac{n}{d}a_{\lambda^d} =0,
		\end{myeq}
		\begin{myeq}\label{eq:au rel Z}
			\frac{d}{\gcd(d,s)}a_{\lambda^s} u_{\lambda^d}=\frac{s}{\gcd(d,s)} a_{\lambda^d}u_{\lambda^s}.
		\end{myeq}
	\end{thm}

For a general cyclic group $G$ and $\alpha \in \bigstar^e$, we observe that the expression above implies that $\pi_\alpha^G(H\uZ)$ is cyclic. This is generated by a product of the corresponding $u$-classes and $a$-classes, and the relation \eqref{eq:au rel Z} implies that they assemble together into a cyclic group. The order of this is the least common multiple of the orders of a product of $a$-classes and $u$-classes occurring in $\pi_\alpha^G(H\uZ)$. 

In $ \Zp $-coefficients the relation  \eqref{eq:au rel Z} simplified to the following 
	\begin{prop} In $ \Zp $-coefficients, we have the following relation
		\begin{myeq}\label{eq:au rel Zp}
			a_{\lambda^{pd}} u_{\lambda^d}=p a_{\lambda^d}u_{\lambda^{pd}}=0.
		\end{myeq}
	\end{prop}
	
	Note that the space $S(\lambda^{p^r})$ fits into a cofibre sequence
	\[ {G/C_{p^r}}_+ \stackrel{1-g}{\to} {G/C_{p^r}}_+ \to S(\lambda^{p^r})_+.\]
	It follows that 
	\begin{myeq}\label{repsph}
	\upi_\alpha^{C_{p^r}} (H\uZ) =0 \mbox{ and } \upi_{\alpha-1}^{C_{p^r}}(H\uZ) = 0 ~ \implies \upi_{\alpha}^G (S(\lambda^{p^r}))=0 .
	\end{myeq} 
	In the case $r=0$, we may make a complete computation to obtain \cite{BG20}
	\begin{myeq}\label{slamb}
	\upi_\alpha(S(\lambda)_+\wedge H\uZ) \cong \begin{cases} \uZ^\ast & \mbox{if } |\alpha|=0,\\
	                 \uZ & \mbox{if } |\alpha|=1 ,\\
	                 0 & \mbox{otherwise}. 
	                 \end{cases}
	 \end{myeq}
	Suppose $\alpha$ satisfies $|\alpha^H|>0$ for all subgroups $H$. This means that $S^\alpha$ has a cell structure with cells of the type $G/H\times D^n$ for $n\geq |\alpha^H|$. Therefore,
	\begin{myeq} \label{grzero}
	 |\alpha^H|>0 \mbox{ for all subgroups } H \implies \upi_\alpha^G(H\uZ)=0.
	 \end{myeq}

	Now if $\alpha$ satisfies $|\alpha^H|<0$ for all subgroups $H$, $\beta = -\alpha$ satisfies the above condition. As $S^\alpha$ is the Spanier-Whitehead dual of $S^\beta$, we may construct it using cells of the type $G/H\times D^n$ for $n<0$. Again, $\upi_\alpha^G(H\uZ)=0$.  Thus,
	\begin{myeq} \label{lezero}
	 |\alpha^H|<0 \mbox{ for all subgroups } H \implies \upi_\alpha^G(H\uZ)=0.
	 \end{myeq} 
	
Using Anderson duality, we extend these relations slightly in the following Proposition.
\begin{prop}\label{mfun0}
		Let $ \alpha\in RO(G) $ be such that $ \alpha $ is even, $ |\alpha |>0 $, and $ |\alpha^K|\ge 0 $  for all subgroups $ K\ne e $. Then  $ \upi_\alpha^G (H\uZ)=0$.
	\end{prop}
	
	\begin{proof}
		By \eqref{and_comp}
		\begin{equation*}
		0 \to  \Ext_L(\upi_{-\alpha+\lambda -3}^{G}(H\uZ), \Z) \to \upi_{\alpha}^{G}(H\uZ) \to \Hom_L(\upi_{-\alpha+\lambda -2}^{G}(H\uZ), \Z) \to 0.
		\end{equation*}
		Since $ |\alpha |>0 $, we have $ |-\alpha+\lambda -2|<0 $. Hence the right hand side term is zero as the Mackey functor  $ \upi_{-\alpha+\lambda -2}^{G}(H\uZ) $ only features torsion elements. Also, $ |(-\alpha+\lambda -3)^H|<0 $ for all $ H\le G $. Hence $ \upi_{\alpha}^{G}(H\uZ)=0 $.
	\end{proof}	

	The following theorem may be viewed as another extension of \eqref{grzero}. It also provides the necessary input in proving homology decompositions. 
	\begin{thm}\label{thm:BG Cn}
		Let $ \alpha\in RO(G) $ be such that $ |\alpha^H| $ is odd for all subgroups $H$, and $|\alpha^H|>-1$ implies $|\alpha^K|\ge -1 $ for all subgroups $ K\supseteq H $. Then $ \upi_{\alpha}^{G}(H\uZ )=0$.
	\end{thm}	
	Note that the second condition stated here may  be equivalently expressed as $|\alpha^H|<-1$ implies $|\alpha^K|\leq -1$ for all subgroups $K$ of $H$. The first condition implies that $\alpha$ does not contain any multiples of the sign representation in the case $|G|$ is even.
	
	\begin{proof}
	It suffices to prove this for $G$ of prime power order, via Proposition \ref{trres}. For groups of odd prime power order, this is proved in \cite{BasuDeyKarmakar-NYJM}.

	Let $\FF_G=\{\alpha\in RO(G) \mid \forall~ H \subseteq G,~ |\alpha^H|>-1 \implies	|\alpha^K|\geq -1 \mbox{ for all } H\subset K \}$. We would like to show that $\alpha\in \FF_G \implies \upi_\alpha(H\uZ)=0$. If $\alpha\in \FF_G$ with $|\alpha^G|<-1$, the hypothesis implies $|\alpha^H|\leq -1$ for all subgroups $H$ of $G$. For these $\alpha$, $\upi_\alpha^G(H\uZ)=0$ by \eqref{lezero} as all the fixed points are negative. 
	
	Now let $\alpha\in \FF_G$, and $H=C_{p^r}$ is a subgroup such that $|\alpha^H|<-1$. This implies that for all $K\subset H$, $|\alpha^K|\leq -1$. For such an $\alpha$, $\alpha-\lambda^{p^s} \in \FF_G$ for $s\leq r$. Also the cofibre sequence 
	\[ S(\lambda^{p^s})_+ \to S^0 \to S^{\lambda^{p^s}},\]
	implies the long exact sequence of Mackey functors 
	\begin{myeq}\label{longex}
	 \upi_{\alpha}^G(  S(\lambda^{p^s})_+\wedge H\uZ) \to \upi_\alpha^G ( H\uZ)\to \upi^G_{\alpha - \lambda^{p^s}}( H\uZ)\to \upi_{\alpha - 1}(  S(\lambda^{p^s})_+\wedge H\uZ).
	 \end{myeq}
	The given condition on $\alpha$ implies that  $\alpha-t$ for $t\in \{0,1,2\}$ have negative dimensional fixed points for subgroups of $C_{p^s}$. It follows from \eqref{repsph} and \eqref{lezero} that $\upi_\alpha^G(H\uZ)\cong \upi_{\alpha-\lambda^{p^s}}(H\uZ)$. In this way by adding and subtracting copies of $\lambda^{p^s}$ for $s\leq r$ while adhering to the condition $|\alpha^K|\leq -1$ for all subgroups $K$ of $H$, we may find a new $\beta \in \FF_G$, satisfying $|\beta^K|=-1$ for all $K=C_{p^s}$ for $s\leq r$, and $\upi_\beta^G(H\uZ)\cong \upi_\alpha^G(H\uZ)$. A consequence of this manoeuvre is that it suffices to prove the result for those $\alpha \in \FF_G$ such that $|\alpha^H|\geq -1$ for all subgroups $H$. Call this collection $\FF_G^{\geq -1}\subset \FF_G$.
	
	A small observation will now allow us to assume $|\alpha|\geq 1$ in $\FF^{\geq -1}_G$. For, we have the long exact sequence  
	\[ \upi_{\alpha }^G(  S(\lambda)_+\wedge H\uZ) \to \upi_\alpha^G ( H\uZ)\to \upi^G_{\alpha - \lambda}( H\uZ)\to \upi_{\alpha - 1}(  S(\lambda)_+\wedge H\uZ),\]
	by putting $s=0$ in \eqref{longex}. Applying the computation of \eqref{slamb}, we deduce $\upi_\alpha^G(H\uZ) \cong \upi_{\alpha-\lambda}^G(H\uZ)$ if $|\alpha|\geq 3$, and if $|\alpha|=1$, $\upi_\alpha^G(H\uZ)=0$ $\implies$ $\upi_{\alpha-\lambda}^G(H\uZ)=0$. The last conclusion is true because for $\nu= \alpha-1$, $|\nu|=0$, and the map 
\[ \uZ^\ast \cong \upi_\nu^G( S(\lambda)_+\wedge H\uZ)) \to \upi_\nu^G(H\uZ),\]
is an isomorphism at $G/e$, and hence injective at all levels. 
	
	Suppose that $\alpha \in \FF_G^{\geq -1}$ and $|\alpha|\geq 1$. By Proposition \ref{trres}, we have $\upi_\alpha^G(H\uZ)$ only features torsion elements as $\upi_\alpha^G(H\uZ)(G/e)=0$. Applying Anderson duality \eqref{and_comp}, we obtain 
	\[ \upi_\alpha^G(H\uZ) \cong \Ext_L(\upi_{\lambda-\alpha-3}^G(H\uZ),\Z).\]
Now note that all the fixed points of $\lambda - \alpha - 3$ are negative if $\alpha \in \FF_G^{\geq -1}$ and $|\alpha|\geq 1$. Therefore, $\upi_\alpha^G(H\uZ)=0$ and the proof is complete. 	
	\end{proof}
	
We also have a calculation of $\upi_\alpha^G(H\uZ)$ if all the fixed points are $\geq 0$. 	
	\begin{prop}\label{prop:new u mfun}
		Let $ \alpha\in RO(G) $ be such that $ |\alpha |=0 $, and $ |\alpha^K|\ge 0 $ even for all subgroups $ K\ne e $. Then  $ \upi_\alpha^G (H\uZ)$ is isomorphic to the Mackey functor $\uZ $.
	\end{prop}
	
	\begin{proof}
	The cofibre sequence 
	\[ S(\lambda)_+\to S^0 \to S^\lambda,\]
	implies the long exact sequence (by taking Mackey functor valued homotopy groups after smashing with $H\uZ$)
	\[ \cdots \upi_{\alpha +\lambda}^G(  S(\lambda)_+\wedge H\uZ) \to \upi_{\alpha+\lambda}^G ( H\uZ)\to \upi^G_{\alpha}( H\uZ)\to \upi_{\alpha+\lambda -1}^G(  S(\lambda)_+\wedge H\uZ)\to \upi_{\alpha+\lambda-1}^G ( H\uZ) \to \cdots\]
	putting $s=0$ in \eqref{longex}. Note that $\alpha+\lambda -1$ satisfies the hypothesis of Theorem \ref{thm:BG Cn}. The element $\alpha + \lambda$ has dimension $2$, so by Proposition \ref{trres}, the Mackey functor $\upi_{\alpha+\lambda}^G(H\uZ)$ features only torsion elements. By Anderson duality \eqref{and_comp}, we have 
	\[\upi_{\alpha+\lambda}^G(H\uZ) \cong \Ext_L( \upi_{-3-\alpha}^G(H\uZ),\Z).\]
Clearly from the given hypothesis, all the fixed points of $-3-\alpha$ are negative, therefore by \eqref{lezero}, $\upi_{\alpha+\lambda}^G(H\uZ)=0$. We obtain
\[\upi_{\alpha}^G(H\uZ) \cong \upi_{\alpha+\lambda -1}^G(S(\lambda)_+\wedge H\uZ) \cong \uZ,\]
by \eqref{slamb}. This completes the proof. 
	\end{proof}
	
	This helps us define the following classes.
	\begin{defn}\label{def:new u-class}
		Let $ j $ be a  multiple of $ i $. Then by Proposition \ref{prop:new u mfun}, the Mackey functor  $ \upi^G_{\lambda^j-\lambda^i}(H\uZ) $ is isomorphic to $ \uZ $.
		Define the class $ u_{\lambda^i-\lambda^j} \in  \pi^G_{\lambda^j-\lambda^i}(H\uZ)$ to be the element which  under restriction to the orbit $ G/e $ corresponds to $1\in \Z $. 
	\end{defn}
The multiplication of the class $ u_{\lambda^k-\lambda^{dk}} $ with $ a_{\lam^{dk}} $ is a multiple of $a_{\lam^k}  $. A similar description also appeared in \cite[p. 395]{HHR-realKthy}.
	\begin{prop}\label{prop:new u cls rel with a cls}  
		We have the following relation
		\[
		u_{\lam^{k}-\lam^{dk}}\, a_{\lam^{dk}}=da_{\lam^k}
		\]
		\[
		u_{\lam^{k}-\lam^{dk}}\, u_{\lam^{dk}}=u_{\lam^k}.
		\]
	\end{prop}
\begin{proof}
	Let us  denote $ a_{{\lambda^{dk}}/{\lambda^k}} $  to be the map 
	$$a_{{\lambda^{dk}}/{\lambda^k}} :S^{\lambda^k}\to S^{\lambda^{dk}},
	$$
	under which  $ z\mapsto z^d $. Therefore, the underlying degree of this map is $ d $. Moreover, 
	$$ a_{{\lambda^{dk}}/{\lambda^k}} \,a_{\lambda^k} =a_{\lambda^{dk}}.
	$$ 
	Hence $ u_{\lam^{k}-\lam^{dk}}\, a_{\lam^{dk}}=u_{\lam^{k}-\lam^{dk}}\, a_{{\lambda^{dk}}/{\lambda^k}}\, a_{\lambda^k}, $
	where the element 
	$$ u_{\lam^{k}-\lam^{dk}}\, a_{{\lambda^{dk}}/{\lambda^k}} \in \uH^0_G(S^0;\uZ)(G/G)\cong \Z.
	$$
	 Since $ \res^G_e( u_{\lam^{k}-\lam^{dk}}) =1$ and $ \res^G_e(a_{{\lambda^{dk}}/{\lambda^k}})=d $, we obtain  $ u_{\lam^{k}-\lam^{dk}}\, a_{\lam^{dk}}=da_{\lam^k} $.
	 
	 Similarly, since $ \res^G_e(u_{\lam^{k}-\lam^{dk}}\, u_{\lam^{dk}})= \res^G_e(u_{\lam^k})=1$, we have $ u_{\lam^{k}-\lam^{dk}}\, u_{\lam^{dk}}=u_{\lam^k}. $
\end{proof}

The following will be used in  subsequent sections.
	\begin{prop}\label{thm:neg htpy gr}
		Let $ \alpha={\lambda^{i_1}+\cdots+\lambda^{i_k}} $. Then the group 
		\[
		H_G^{-\alpha} (S^0)=0 .
		\]
	\end{prop}
\begin{proof}
	 For a representation $ \lambda^{i_s} $, we have the cofibre sequence $ S(\lambda^{i_s})_+\xrightarrow{r} S^0\to S^{\lambda^{i_s}} $. If $ H_s $ is the kernel of the representation  $ \lambda^{i_s} $, then we have the cofibre sequence  $ {G/H_s}_+  \stackrel{1-g}{\to} {G/H_s}_+\to S(\lambda^{i_s})_+$. To see $ H_G^{-\lambda^{i_1}} (S^0)=0$, consider the long exact sequence 
	\[
0\to H^{-1}_G(S(\lambda^{i_1})_+)\to H_G^{-\lambda^{i_1}}(S^0)\to 	H^0_G(S^0)\xrightarrow{r^*} H_G^0(S(\lambda^{i_1})_+)\to \cdots 
	\]
	The first term is zero  as $H^{j}_G({G/H_{1}}_+)=0  $ for $ j\le -1 $. Also,  $ H_G^0(S(\lambda^{i_1})_+)\cong  H^{0}_G({G/H_{1}}_+)\cong \Z$. Moreover, under this identification, the map $ r^* $ is the restriction map $ res^G_{H_1} (\uZ)$, hence $ r^* $ is an isomorphism. Thus $ H_G^{-\lambda^{i_1}} (S^0)=0$. Using similar arguments, the result  follows by induction. 
\end{proof}
\begin{exam}\label{exam S sigma }
	 The Mackey functor $ \underline{H}_{C_{2}}^{0}(S^\sigma;\uZ)$ is zero. To see this consider the cofibre sequence $${C_{2}/e}_{+}\ra S^0 \ra S^{\sigma}$$ and the associated long exact sequence in cohomology 
	$$0\to  \underline{H}^{0}_{C_{2}}(S^{\sigma};\uZ)\ra \underline{H}^{0}_{C_{2}}(S^0;\uZ)\ra \underline{H}^{0}_{C_{2}}({C_{2}/e}_{+};\uZ)\ra \uH_{C_2}^{1}(S^{\sigma};\uZ)\ra \dots $$
	The restriction map $\underline{H}_{C_2}^{0}(S^0;\uZ)\ra \underline{H}_{C_2}^{0}({C_{2}/e}_{+};\uZ) $ is injective.
	Thus we have $\underline{H}_{C_2}^{0}(S^{\sigma};\uZ)= 0$.
\end{exam}
	Next we note a homology decomposition theorem for a cyclic group by generalizing Lewis's approach. We say that a representation $W$ is even if all the fixed points of $W$ are even dimensional.

	\begin{thm}\label{thmlewis}
		Let $X$ be a finite type (that is, with finitely many cells of each dimension) generalised $G$-cell complex with only even dimensional cells of the form $D(W)$. Suppose further that for cells $D(W), D(V)$ we have the condition 
\[\dim W^H<\dim V^H \implies |W^K|\leq|V^K| \mbox{ for every subgroup } K \mbox{ containing }H.\]
 Let $\CC$ denote the collection of cells of $X$ under the above description. Then,
 $$H\uZ\smas X_+ \simeq H\uZ\vee \bigvee_{D(W)\in \CC}H\uZ \smas S^{W}.$$  
	\end{thm}
	\begin{proof}
		The main step of the proof involves a pushout diagram of the form \begin{myeq}\label{orb-conn-sum-cof}
			\xymatrix{ S(V) \ar[r] \ar[d] & D(V) \ar[d] \\ 
				X \ar[r] & Y }
		\end{myeq}
		where we know that $H\uZ\smas X\simeq\bigvee_{i=1}^{k} H\uZ\smas S^{W_i}$. Look at the cofibre sequence 
$$H\uZ\smas X\to H\uZ\smas Y\to H\uZ\smas S^V.$$ 
		The connecting map goes from $H\uZ\smas S^{V}$ to $\bigvee_{i=1}^{k}H\uZ\smas S^{W_{i} + 1}$. For each $i$, $H\uZ\smas S^{V}\ra H\uZ\smas S^{W_{i} + 1}$ is in $\underline{\pi}_{\a}^{G}(H\uZ)$ where $\a = V - W_{i}-1$. Using Theorem \ref{thm:BG Cn} for this $\a$ we get the connecting map to be $0$. Hence the result follows.
	\end{proof}
	\section{Additive homology decompositions for projective spaces}\label{sec:additive}
	In this section,  we show that $H\uZ\smas P(V)$ is a wedge of suspensions of $H\uZ$ in many examples. Along the way, we also construct suitable bases for the homology which are used in later sections.
	\begin{mysubsection}{Cellular filtration of complex projective spaces}
		For a complex representation $V$ of $G$, the equivariant complex projective space $ P(V) $ is the set of complex lines in $V$. It is constructed by attaching even dimensional cells of the type $ D(W) $
		for representations $W$. We note that $P(V)$ and $P(V\otimes \phi)$ are homeomorphic as $G$-spaces for a one dimensional $G$-representation $\phi$. As $G$ is Abelian, the complex representation $V$ is a direct sum of $\phi_{i}$ where $\dim(\phi_{i})= 1$. If we write $V=V'\oplus\phi$ for a one-dimensional representation $\phi$, we have a cofibre sequence 
		$$P(V')\ra P(V)\ra S^{\phi^{-1}\otimes V'}.$$
		As a consequence, we obtain a cellular filtration of $P(V)$, which we proceed to describe now. Write $V=\phi_{1}+\phi_{2}+\dots+\phi_{n}$ and $V_{i}= \phi_{1}+\phi_{2}+\dots+\phi_{i}$. The cellular filtration is given by $$P(V_{1})\subseteq P(V_{2})\subseteq \dots \subseteq P(V_{n})= P(V)$$ with cofibre sequences
		$$P(V_{i})\ra P(V_{i+1})\ra S^{\phi_{i+1}^{-1}\otimes V_{i}}.$$
		This shows that $P(V_{i+1})$ is obtained from $P(V_{i})$ by attaching a cell of the type $D(W_{i})$ for $ W_{i} = \phi^{-1}_{i+1}\otimes V_{i}$.
		Note that this filtration depends on the choice of the ordering of the $ \phi_i $.
Via Theorem \ref{thmlewis}, we try for decompositions of $H\uZ\smas P(V)$ as a wedge of suspensions of $H\uZ$. 
	\end{mysubsection}

	\begin{exam}
	For $G=C_p$ for $p$ odd, consider $P(V)$ for $V = \lambda^0 +2 \lambda + \lambda^2$. We may write $V=\lam^{0}+\lam + \lam + \lam^2$ and obtain the corresponding cellular filtration on $P(V)$. The corresponding cells are $D(W_{m})$ for $W_{m} = \phi_{m}^{-1}\otimes V_{m-1}$ for $m\leq 4$. Observe that $|W_3| < |W_4|$ but $|W_{3}^{C_p}| =2 > 0 = |W_{4}^{C_p}|$ which means that the hypothesis of Theorem \ref{thmlewis} is not satisfied. However a simple rearrangement allows us to write down a homology decomposition. Write $V= \lam^0 + \lam +\lam^2 +\lam$, and we now see that the resulting $W_{i}$ satisfy the hypothesis of Theorem \ref{thmlewis}. This implies 
		$$H\uZ\smas P(V)_+\simeq H\uZ\bigvee H\uZ\smas S^{\lambda}\bigvee H\uZ\smas S^{2\lambda}\bigvee H\uZ\smas S^{2\lambda + 2}.$$
	\end{exam}
	
	\vspace{.5cm}
	
	In the following theorem, let $V =  n_0\lambda^0 + n_1\lambda^1 + \dots + n_{p-1}\lambda^{p-1} $ be any $C_p$-representation. Except for the fact that the $n_i$'s are non-negative, no other condition is imposed on $n_{i}$. We may assume $n_{0}\geq n_{i}$ by replacing $V$ with $V\otimes\lam^{-j}$ if necessary. 
	
	\begin{thm}\label{thm V with uneven coeff}
		Let $V=  n_0\lambda^0 + n_1\lambda^1 + \dots + n_{p-1}\lambda^{p-1}$ be a complex $C_p$-representation and $n_0\geq n_i\geq 0$  for all $ i $. Then 
		$$H\uZ\smas P(V)_+ \simeq H\uZ\bigvee_{i=1}^{a_1-1}\Sigma^{i\lambda} H\uZ\bigvee_{i=a_1-1}^{a_1+a_2-2} \Sigma^{i\lambda+2}H\uZ\bigvee \dots \bigvee_{i=(\sum_{j=1}^{n_0-1}a_{j})- (n_0 - 1)}^{(\sum_{j=1}^{n_0}a_{j})- n_0 } \Sigma^{i\lambda + 2(n_0-1)}H\uZ$$ 
		where  $a_i$ is the cardinality of the set $\{ n_j: n_j\geq i\} $.
	\end{thm}
	
	\begin{proof}
		We arrange the irreducible representations in $V$ in such a way  that 
		$$V = A_1 + A_2 + \dots + A_{n_0}$$
		where $A_1 = \sum_{n_i \geq 1} \lambda^i, A_2 = \sum_{n_i\geq 2}\lambda^i, \dots , A_{n_0} = \sum_{n_i\geq n_0}\lambda^i$. Then, $a_i$ is the number of summands appearing in $A_i$.
		
		We consider the cell complex structure on $P(V)$ associated to $V=\sum_{i=1}^{\dim_{\C}V}\phi_i$ where $\phi_i$'s are defined by 
		$$\phi_{(\sum_{l=1}^{j-1} a_{l}) +1}+\phi_{(\sum_{l=1}^{j-1} a_{l})+2}+ \dots + \phi_{(\sum_{l=1}^{j-1} a_{l}) + a_{j}} = A_{j} = \sum_{n_i\geq j}\lambda^i $$
		for $j\geq 1$, assuming $a_{0}=0 $, and the powers of $\lam$ above are arranged in increasing order from $0$ to $p-1$.
		To prove the statement, we use induction on the sum $n_0 + n_1 +\dots + n_{p-1} = \dim_{\C}V$. 
		
		When $n_0 + \dots + n_{p-1}=1$, that is, $n_0 = 1$ and $n_i=0 $ $ \forall~ 1\leq i \leq p-1$, then $V= \lambda^0$. Thus $P(V)_+ = S^0$ and 
		$$H\uZ\smas P(V)_+ \simeq H\uZ\smas S^0 \simeq H\uZ.$$
		Now suppose that the statement is true for integers less than $n_0 + n_1 + \dots + n_{p-1}$. Using the notation $V_k =\sum _{i=1}^k\phi_{i} $ as above, the inductive hypothesis implies the result for $X=P(V_k)$  whenever $k < \dim_{\C}V$. In particular, letting $m=\sum n_i - 1$,
		$$V=V_m + \phi_{\dim_{\C}V} = V_m + \lambda^s $$
		for some integer $s$. Let $a'_{i}, n'_{i}$ and $A'_{i}$ denote the values for $V_m$ that correspond to $a_{i}, n_{i}$ and $A_{i}$ for $V$. Observe that $a'_{i} = a_{i}$ if $i<n_0$, $a'_{n_{0}} = a_{n_{0}} - 1$, and our choice of the $\phi_{i}$ implies that $n_{s} = n_{0}$. The induction hypothesis implies 
		\begin{myeq}\label{decforVm}
			{H\uZ\smas P(V_m)_+ \simeq H\uZ\bigvee_{i=1}^{a'_1-1}\Sigma^{i\lambda} H\uZ\bigvee_{i=a'_1-1}^{a'_1+a'_2-2} \Sigma^{i\lambda+2}H\uZ\bigvee \dots \bigvee_{i=(\sum_{j=1}^{n'_0-1}a'_{j})-(n'_0-1)}^{(\sum_{j=1}^{n'_0}a'_{j})- n'_0 } \Sigma^{i\lambda + 2(n'_0-1)}H\uZ}.
		\end{myeq}
		We see that, either $s=0$ or if $s\neq 0$, $n_{s}=n_{0}$. We first consider the latter case. Then $n'_0 = n_0$, $a'_i=a_i$ whenever $i< n_0$ and $a'_{n_0}= a_{n_0}-1$. Thus \eqref{decforVm} reduces to 
		$$H\uZ\smas P(V_m)_+ \simeq H\uZ\bigvee_{i=1}^{a_1-1}\Sigma^{i\lambda} H\uZ\bigvee_{i=a_1-1}^{a_1+a_2-2} \Sigma^{i\lambda+2}H\uZ\bigvee \dots \bigvee_{i=(\sum_{j=1}^{n_0-1}a_{j})-(n_0-1)}^{(\sum_{j=1}^{n_0}a_{j})-( n_0 +1) } \Sigma^{i\lambda + 2(n_0-1)}H\uZ.$$
		
		As $n_0 = n_s$, the coefficient of $\lambda^s$  in  $V_m = V-\lambda^s$  is  $n_s - 1 = n_0 -1$, which in turn implies the coefficient of  $\lambda^0$   in  $\lambda^{-s}\otimes V_m$  is $n_0 -1$. Thus $ |(\lambda^{-s}\otimes V_m)^{C_p}| = 2(n_0 - 1)$. We look at representations $i\lambda +j\lambda^0$ where $i \in \{0,1, \dots , a_0+a_1+\dots + a_{n_0 -1} -1- n_0\}$ and $j \in \{0,1,\dots , n_0 -1 \}$. Then,
		$$|(i\lambda + j\lambda^0)^{C_p}| = 2j \leq 2(n_0 - 1) = |(\lambda^{-s}\otimes V_m)^{C_p}|.$$
		
		When $s = 0 $ we have $ n'_0 = n_0-1$, $a'_i = a_i$ for all $i<n'_0$. In this case, \eqref{decforVm} reduces to,
		$$H\uZ\smas P(V_m)_+ \simeq H\uZ\bigvee_{i=1}^{a_1-1}\Sigma^{i\lambda} H\uZ\bigvee_{i=a_1-1}^{a_1+a_2-2} \Sigma^{i\lambda+2}H\uZ\bigvee \dots \bigvee_{i=(\sum_{j=1}^{n_0-2}a_{j})-(n_0-2)}^{(\sum_{j=1}^{n_0-1}a_{j})- (n_0-1) } \Sigma^{i\lambda + 2(n_0-2)}H\uZ.$$  
		Note that $a_{n_0}= 1 $ that  is $a_{n_0}-1=0$. So  we  can rewrite the equation as 
		$$H\uZ\smas P(V_m)_+ \simeq H\uZ\bigvee_{i=1}^{a_1-1}\Sigma^{i\lambda} H\uZ\bigvee_{i=a_1-1}^{a_1+a_2-2} \Sigma^{i\lambda+2}H\uZ\bigvee \dots \bigvee_{i=(\sum_{j=1}^{n_0-2}a_{j})-(n_0-2)}^{(\sum_{j=1}^{n_0}a_{j})- n_0 }\Sigma^{i\lambda + 2(n_0-2)}H\uZ.$$
		Since $s=0, \lambda^{-s}\otimes V_m = V_m$. Now for all representations of the form $i\lambda + j\lambda^0$ where $i \in \{0,1,\dots, (a_0+a_1+\dots + a_{n_0 - 1}-n_0) \}$ and $j \in \{0,1,\dots,(n_0 - 2)\}$ we have 
		$$|(i\lambda + j\lambda^0)^{C_p}| = 2j \leq 2(n_0-2) < 2(n_0 -1) = |V_{m}^{C_p}| = |(\lambda^{-s}\otimes V_m)^{C_p}|.$$
		These calculations imply that both cases satisfy the hypothesis of Theorem \ref{thmlewis}. Thus, we obtain  
		$$H\uZ \smas P(V)_+ \simeq H\uZ \smas P(V_m)_+\bigvee H\uZ \smas S^{\lambda^{-s}\otimes V_m}.$$
		Using the facts $H\uZ\smas S^{\lambda^i} \simeq H\uZ \smas S^{\lambda}$ and $H\uZ\smas S^{\lambda^0} \simeq H\uZ \smas S^2$, 
		$$H\uZ\smas P(V)_+ \simeq  H\uZ\smas P(V_m)_+ \bigvee H\uZ\smas  S^{(\sum_{j=1}^{n_{0}}a_j - n_0) \lambda + 2(n_0 -1)}$$ 
		for both $s=0$ and $s\neq 0$.
	\end{proof}
	
	We elaborate this theorem with an example below. 
	\begin{exam}
		Let $V = 3\lambda^0+ 2\lambda + 4\lambda^2$ be a complex $C_p$-representation. Since among the three coefficients appearing in the expression for $V$ here, $4$ is the largest, we consider $V\otimes\lambda^{-2} = 4\lambda^0 + 3\lambda^{p-2} + 2 \lambda^{p-1}$. For simplicity, we call this also as $V$. We now write $V$ as a sum of $A_i$, that is, $V= A_1 + A_2 + A_3 + A_4$, where $A_1 = \lambda^0 + \lambda^{p-2} + \lambda^{p-1}$, $A_2 = \lambda^0 + \lambda^{p-2} + \lambda^{p-1}$, $A_3 = \lambda^{0}+\lambda^{p-2}$, $A_4 = \lambda^0$.
		
		Note that $a_1 = 3, a_2 = 3, a_3 = 2$ and $a_4 = 1$. From Theorem \ref{thm V with uneven coeff}, we conclude that 
		$$H\uZ \smas P(V)_+ \simeq H\uZ \smas P(V_m)_+\bigvee H\uZ \smas S^{V_m}$$  which is same as  
		$$H\underline{\mathbb{Z}}\bigvee_{i=1}^{2}\Sigma^{i\lambda} H\uZ\bigvee_{i=2}^{4} \Sigma^{i\lambda+2}H\uZ \bigvee_{i=4}^{5 } \Sigma^{i\lambda + 4}H\uZ\bigvee \Sigma^{5 \lambda + 6}H\uZ.$$
	\end{exam}

\begin{mysubsection}{Decompositions over general cyclic groups}
	We now proceed towards the equivariant homology decomposition of complex projective spaces where the group $G$ is any cyclic group $C_n$. Note that a complete $C_n$-universe $\sU$ may be constructed as 
$$\sU = \mathop{\lim_{\longrightarrow}}_{m} m\rho$$ 
where $\rho$ is the regular representation. In the remaining part of the section, we stick to these representations to avoid the involved expressions as in Theorem \ref{thm V with uneven coeff} for the general cyclic groups.
	\begin{thm}\label{thm main simple}
		We have the decomposition 
		$$H\uZ\smas P(m\rho)_+ \simeq \bigvee_{i=0}^{nm-1}H\uZ\smas S^{\phi_{i}}$$ 
		where $\phi_{0}= 0$ and $\phi_{i} =\lam^{-i}(1_\C+ \lam + \lam^2 +\dots + \lam^{i-1})$ for $i> 0$. Passing to the homotopy colimit, for a $C_n$-universe $\sU$, we obtain 
		$$H\uZ\smas P(\sU)_+ \simeq \bigvee_{i=0}^{\infty}H\uZ\smas S^{\phi_{i}}.$$
	\end{thm}
	
	\begin{proof}
		We use induction on $k$ to show that 
		$$H\uZ\smas P(V_{k})_{+}\simeq \bigvee_{i=0}^{k}H\uZ\smas S^{\phi_{i}}$$
		for  $V_{k}=\sum_{i=0}^{k}\lam^{i}$. 
		The statement holds for $k=0$ since $V_{0}= \lam^0 = 1_{\C}$ and $H\uZ\smas P(V_{0})_+ \simeq H\uZ\smas S^0$. 
		
		Now let the statement be true for $V_{k}$ i.e $H\uZ\smas P(V_k)_+ \simeq \bigvee_{i=0}^{k} H\uZ\smas S^{\phi_{i}}$. We have the cofibre sequence 
		$$P(V_k)_+\ra P(V_{k+1})_+\ra S^{\phi_{k+1}}.$$
		Smashing with $H\uZ$ we have 
		$$H\uZ\smas P(V_k)_+\ra H\uZ \smas P(V_{k+1})_+\ra H\uZ \smas S^{\phi_{k+1}}.$$
		Thus the connecting map goes from $H\uZ\smas S^{\phi_{k+1}}$ to $\bigvee_{i=0}^{k} H\uZ\smas S^{\phi_{i}+1}$. For each $i\leq k$ the map $H\uZ\smas S^{\phi_{k+1}}\ra H\uZ\smas S^{\phi_{i}+1}$ belongs to $\pi_{0}^{C_{n}}(H\uZ\smas S^{\phi_{i}+1-\phi_{k+1}})\simeq \pi_{\phi_{k+1}-\phi_{i}-1}^{C_n}(H\uZ)$. Taking $\a =\phi_{k+1}-\phi_{i}-1$, we have $|\a|$ is odd. Note that 
		$$|\a^{H}| = |\phi_{k+1}^{H}|-|\phi_{i}^{H}|-1= |(\sum_{j=i+1}^{k+1}\lam^{j})^{H}| - 1 \geq -1.$$ 
		Applying Theorem \ref{thm:BG Cn} we have that the connecting map is $0$. Thus 
		$$H\uZ\smas P(V_{k+1})_+ \simeq \bigvee_{i=0}^{k+1} H\uZ \smas S^{\phi_{i}}.$$
		Hence the theorem follows.
	\end{proof}
	
	We may also consider a variant in the case $G=C_{2}$ which acts on $\C P^{n}$ by complex conjugation. The resulting $C_{2}$-space is denoted ${\C P}_{\tau}^n$ for $1\leq n \leq \infty$. Then, the fixed points $(\C P^{n}_{\tau})^{C_{2}} \simeq \mathbb{R} P^n$ which shows that this example is homotopically different from the example above.
	
	\begin{thm}\label{prop:cpnconj}
		We have the decomposition 
		$$H\uZ \smas {\C P_{\tau}^{n}}_+  \simeq \bigvee_{i=0}^{n}H\uZ\smas S^{i+i\sigma }\simeq \bigvee_{i=0}^{n}H\uZ\smas S^{i\rho}.$$ 
		Passing to the homotopy colimit we have 
		$$H\uZ \smas {\C P_{\tau}^{\infty}}_+  \simeq \bigvee_{i=0}^{\infty}H\uZ\smas S^{i+i \sigma }\simeq \bigvee_{i=0}^{\infty}H\uZ\smas S^{i\rho}.$$
	\end{thm}
	
	\begin{proof}
		The method of the proof is exactly same as \ref{thm main simple}. The main step involves the cofibre sequence 
		$$ {\C P_{\tau}^{n-1}}_+  \hookrightarrow  {\C P_{\tau}^{n}}_+  \ra S^{n+n\sigma }.$$
		Smashing with $H\uZ$ gives
		$$ H\uZ\smas{\C P_{\tau}^{n-1}}_+  \ra H\uZ\smas{\C P_{\tau}^{n}}_+  \ra H\uZ\smas S^{n+n\sigma }.$$
		The connecting homomorphism goes from 
		$$H\uZ\smas S^{n+n\sigma}\ra \bigvee_{i=0}^{n-1}H\uZ\smas S^{i+i\sigma +1}.$$  
		as $ H\uZ \smas {\C P_{\tau}^{n-1}}_+  \simeq \bigvee_{i=0}^{n-1}H\uZ\smas S^{n+n \sigma }$, by the inductive hypothesis. For each $0\leq i\leq n-1$, up to homotopy, the map $H\uZ\smas S^{n+n\sigma}\ra H\uZ\smas S^{i+i\sigma +1}$ belongs to $\pi_{0}(S^{i+1-n+(i-n)\sigma}\smas H\uZ)\cong H^{{i+1-n+(i-n)\sigma}}_G(S^0)$, which is $0$ when $i<n-1$ as all the fixed-point dimensions are negative.
		 At $i=n-1$, we get the Mackey functor $\underline{\pi}_{0}(S^{-\sigma}\smas H\uZ) = \underline{H}_{C_{2}}^{-\sigma}(S^0;\uZ)$, which is zero by Example \ref{exam S sigma }.
		 Hence the Theorem follows.
	\end{proof}
	\end{mysubsection}	
	
	\begin{mysubsection}{Quaternionic projective spaces}
		As in the complex case, the quaternionic projective spaces may be equipped with a cell structure that turn out to be useful from the perspective of homology decompositions \cite{Chonoles}. The quaternionic $C_n$-representation $\psi^r$ is given as multiplication by $e^{\frac{2\pi ir}{n}}$ on $\mathbb{H}$.  As a complex $C_n$-representation, $\psi^{r} \cong \lambda^{r} + \lambda^{-r}$. The equivariant projective space $P_{\mathbb{H}}(V)$ for a quaternionic representation $V$, is the set of lines in $V$, that is, $P_{\mathbb{H}}(V)= V \setminus \{0\} / \sim$ where $v \sim hv$ $ \forall~ v\in V\setminus \{0\}$, and $ \forall~ h \in \mathbb{H} \setminus \{0\}$. Define $\rho_\Hyp = \Hyp\otimes_\C \rho_G$, and note that as $\Hyp$-representations, 
\[\rho_\Hyp= \sum_{i=0}^{n-1} \psi^i.\]
We write $V_{k}= \sum_{i=0}^{k-1}\psi^i$ and $W_k=\lambda^{-k}\otimes_\C V_k$. We recall from \cite{Chonoles} that $P_\Hyp(m\rho_\Hyp)$ is a $G$-cell complex with cells of the form $D(W_k)$ for $k \leq mn-1$. 
	\end{mysubsection}
	
	\begin{thm}\label{quatadd}
Let $G =C_n$. We have the splitting 
$$H\uZ \smas P_{\mathbb H}(m\rho_\Hyp)_+ \simeq \bigvee_{i=0}^{mn-1}H\uZ\smas S^{W_i}.$$
Passing to the homotopy colimit, we obtain 
 $$H\uZ\smas B_GS^3_+ \simeq H\uZ \smas P_{\mathbb H}(\UU_\Hyp)_+ \simeq \bigvee_{i=0}^{\infty}H\uZ\smas S^{W_i}.$$
	\end{thm}
	
	\begin{proof}
		We proceed as in Theorem \ref{thm main simple} by showing via induction on $k$ that 
		$$H\uZ\smas P(V_{k})_+\simeq \bigvee_{i=0}^{k-1}H\uZ\smas S^{W_{i}}.$$
		Along the way we are required to check $|W_{k}^{H} - W_{i}^{H} - 1|\geq -1 $ for $i<k$ which implies Theorem \ref{thm:BG Cn} applies to prove the result. Indeed, it is true as 
\[|W_{i}^{C_d}|=\begin{cases} \lfloor{\frac{2i-1}{d}}\rfloor +1 & \mbox{if } d\mid i \\ 
                        \lfloor{\frac{2i-1}{d}}\rfloor &\mbox{if } d \nmid i, \end{cases}\]
 is a non-decreasing function of $i$. Hence our result follows.  
	\end{proof}
	
	\begin{mysubsection}{Construction of a homology basis}\label{cons:alpha phi d}
We now define the classes $ \alpha_{\phi_\ell} $  which  serve as additive generators of $ H^{\bigstar}_{G}(P(\sU)) $.
		Let $ W_\ell$  and $ \phi_\ell $ be the representations
		\begin{myeq}\label{eq:phi d}
			W_\ell:=1_\C+\lambda+\cdots+\lambda^{\ell} \quad \textup{~and ~} \quad \phi_\ell:=\lambda^{-\ell}(1_\C+\lambda+\cdots+\lambda^{\ell-1}).
		\end{myeq}
		Consider the cofibre sequence 
		$$
		P(W_{\ell-1})_+\hookrightarrow P(W_\ell)_+ \xra{\chi}S^{\phi_\ell}.
		$$
		At $ \deg \phi_\ell $, the associated long exact sequence  is 
		$$
		\cdots \to \tH^{\phi_{\ell}-1}_{G}(P(W_{\ell-1})_+)\ra \tH^{\phi_\ell}_{G}(S^{\phi_\ell})\xra{\chi^*} \tH^{\phi_\ell}_{G}(P(W_\ell)_+)\ra  \tH^{\phi_\ell}_{G}(P(W_{ \ell-1})_+)\ra \cdots $$
		Note that, $ \uH^{\phi_{\ell}-1}_{G}(P(W_{\ell-1})_+)(G/e) \cong H^{2\ell-1}( P({\C^\ell}))\cong 0$.
		So, the  restriction of the map $ \chi^* $ at the orbit $ G/e $ is an isomorphism. Hence, the Mackey functor diagram says that the image of $ 1\in \tH^{\phi_\ell}_{G}(S^{\phi_\ell})\cong  H^{0}_{G}(\pt) \cong \Z  $ under the map $ \chi^* $ is nonzero. Define $ \alpha_{\phi_\ell}^{W_\ell} $  to be the element $ \chi^*(1) $. We often omit the superscript  and simply write  $ \alpha_{\phi_\ell} $.\par
		Now we lift  $\a_{\phi_\ell}$ by induction to get the  generator  $ \alpha_{\phi_\ell}^\mathcal{U} $ (or simply $ \alpha_{\phi_\ell} $) which belongs to  $ H^{\phi_\ell}_{G}(P(\sU)) $. For this we successively add representations (one at a time)  to $ W_\ell $ in a proper order to reach $ \mathcal{ U} $.  Let $ U' \subseteq \mathcal{U}$ be a representation containing $ W_\ell $. Assume that for  $U'$ the class $\alpha_{\phi_\ell}$ has been defined for $H^{\phi_\ell}_{G}(P(U'))$. Suppose $ U=U'+\lambda^j $.
		Consider the cofibre sequence $$P(U')_+ \xrightarrow{\theta}P(U)_+\rightarrow S^{\lam^{-j} U'}$$  and thus the long exact sequence
		\begin{myeq}\label{eq:les}
			\cdots\ra \tH^{\phi_\ell}_G(S^{\lam^{-j} U'})\ra \tH^{\phi_\ell}_G(P(U)_+)\xra{\theta^*}\tH^{\phi_\ell }_G(P(U')_+)\ra \tH^{\phi_\ell +1}_{G}(S^{\lam^{-j} U'})\to\cdots
		\end{myeq}
		By Proposition   \ref{thm:neg htpy gr} and \ref{thm:BG Cn}, the first and the fourth term in \ref{eq:les} is zero. So the map $ \theta^* $ is an isomorphism. Hence a unique lift of $\a_{\phi_\ell}$ exists in $H^{\phi_\ell}_{G}(P(U))$ along the map 
		$$\theta^*:\tH^{\phi_\ell}_{G}(P(U)_+)\ra \tH^{\phi_\ell}_{G}(P(U')_+).
		$$
		Since the restriction of the map $ \chi^* $ to the orbit $ G/e $  is an isomorphism, we get 
		\begin{myeq}\label{eq:res g/e}
			\res^{G}_e(\alpha_{\phi_\ell})=x^{\ell}.
		\end{myeq}
	\end{mysubsection}
	\begin{mysubsection}{Additive generators of $H^{\bigstar}_{G}(P(\mathcal{U}))$}
By Theorem \ref{thm main simple}, we may express the additive structure of the cohomology of $P(\sU)$  as
		\begin{myeq}\label{eq:addi decom}
			H^{\bigstar}_{G}(P(\sU)_{}) = \bigoplus ^{\infty}_{k=0}\bigoplus ^{n-1}_{i=0} H^{\bigstar - k\phi_n - \phi_i}_G(\pt),
		\end{myeq}
		where $ \phi_i=\lambda^{-i}(1_\C+\lambda+\cdots+\lambda^{i-1}) $ and $ \phi_0=0 $. The above construction  defines the  generators $ \alpha_{k\phi_n+\phi_i} \in H^{k\phi_n+\phi_i}_{G}(P(\sU)_{})$ which corresponds to the factor $ H^{\bigstar - k\phi_n - \phi_i}_G(\pt) $ in \eqref{eq:addi decom}.  Summarizing the above, we get
		\begin{prop}
			Additively, the classes $ \alpha_{k\phi_n+\phi_i}$ generates $  	H^{\bigstar}_{G}(P(\sU)_{})$   as a module over $H^{\bigstar}_{G}(\pt)$.
		\end{prop}
	\end{mysubsection}
	
	\section{Application for cohomology operations} \label{sec cohop}
	Let $ G=C_p $ ($ p $ not necessarily odd).  Recall that $ \mathcal{A}_p $ is the mod $ p $ Steenrod algebra. We consider 
	$$
	\mathcal{A}_{G}^n= \{H\Zp, \Sigma^n H\Zp\}^G, 
	$$
	for $ n\in \Z $, and the map
	$$ \Omega:  \mathcal{A}_{G}^*\to \mathcal{ A}_p, $$
	as its restriction to the identity subgroup.
	We demonstrate that the additive decomposition of \S \ref{sec:additive} recovers the following result of Caruso \cite{Caruso}. 
	\begin{thm}\label{thm:coh op}
		Let $ \theta $ be a  degree $ r $ ($ r $ is even) cohomology operation  not involving the B\"{o}ckstein $ \beta $. For such $ \theta $, there does not exist an equivariant cohomology operation 
		\[
		\tilde{\theta}: H\Zp\to \Sigma^{r}H\Zp
		\] 
		such that $ \Omega(\tilde{\theta})=\theta $.
	\end{thm}
	Before going into the proof, let us look at  some examples. 
	\begin{exam}\label{ex:coh op}
		Let $ p $ be odd. We claim that there does not exist an equivariant cohomology operation
		\[
		\tilde{P^1}: H\Zp\to \Sigma^{2p-2}H\Zp
		\] 
		such that $ \Omega(\tilde{P^1}) =P^1$, where $ P^1 $ is the power operation.
		The existence  of  such a   $ \tilde{P^1} $ will lead to a map of Mackey functors 
		$$ \uH_G^\alpha(X;\Zp) \to \uH_G^{\alpha+2p-2}(X;\Zp) \quad \textup{~for~}  \alpha\in RO(G),
		$$ 
		which is natural in $ X $. In particular, let us take $ X=P(\mathcal{U}) $ and $ \alpha =\lambda $. We observe that 
		\begin{myeq}\label{diag:cohop}
			\xymatrix@R=0.5cm@C=0.5cm{\uH_G^\lambda(P(\mathcal{U});\Zp) \ar[d]^{\cong}\ar[rr]^-{\tilde{P^1}} 
				&&\uH_G^{\lambda+2p-2}(P(\mathcal{U});\Zp)\ar[d]^{\cong}
				\\ 
				\Zp & &\Zp^*.
			}
		\end{myeq}
		An explanation for this is as follows. For the group $ C_p $, the additive decomposition in Theorem \ref{thm main simple} tells us
		\begin{myeq}\label{eq:un4}
			P(\mathcal{U})\smas H\uZ\simeq \bigwedg_{k=0}^{\infty} \bigwedg_{i=0}^{p-1} S^{k\phi_p+i\lambda}\wedge H\uZ \qquad \textup{~such that~}(k,i)\ne (0,0).
		\end{myeq}
		As a result,
		\begin{equation*}
			\tH^{\bigstar}_{G}({P(\mathcal{U})};\Zp)\cong \bigoplus_{k=0}^{\infty}\bigoplus_{i=0}^{p-1} \tH_G^{\bigstar-k\phi_p-i \lambda}(S^0;\Zp),\qquad (k,i)\ne (0,0).
		\end{equation*}
		Let $ \alpha=\lambda+2p-2-k\phi_p-i\lambda $. So $ |\alpha|=2p-2kp-2i $ and $ |\alpha^{C_p}|=2p-2-2k$.  Applying \eqref{cp-comp zp cof}, we conclude that $\uH_G^{\lambda+2p-2}(P(\mathcal{U});\Zp)\cong \Zp^* $.\par
		As $ P^1(x)=x^p $ for a generator $ x\in H^2(\C P^\infty) $, the diagram \ref{diag:cohop} yields a commutative diagram  
		\begin{equation*}\label{diag:mfuncohop}
			\xymatrix@R=0.1cm{  \Z/p \ar@/_.5pc/[ddd]_{1}\ar[rr]^{ \tilde{P^1}} && \Z/p \ar@/_.5pc/[ddd]_{0}
				\\
				\\
				\\
				\Z/p\ar@/_.5pc/[uuu]_{  0  } \ar[rr]^{\cong} & & \Z/p \ar@/_.5pc/[uuu]_{  1 },
			}
		\end{equation*}
		which is a contradiction.
	\end{exam}
	The technique used in example \ref{ex:coh op} does not work when $ p=2 $. This is because the Mackey functor $ \Zt $ may now appear  in the right hand side of the diagram \ref{diag:cohop}, and so the contradiction drawn out by comparing the Mackey functor diagram fails. Below we argue differently to show that $ Sq^2 $ is not in the image of $ \Omega $.
	\begin{exam}\label{ex:cohop2coeff}
		Let $ p=2 $. There does not exist an equivariant cohomology operation
		\[
		\tilde{Sq^2}: H\Zt\to \Sigma^{2}H\Zt
		\] 
		such that $ \Omega(\tilde{Sq^2}) =Sq^2$. The existence of such will lead to a map of Mackey functors 
		$$ \uH_G^\alpha(X;\Zt) \to \uH_G^{\alpha+2}(X;\Zt) \quad \textup{~for~}  \alpha\in RO(G).		$$  
Let $ X=\C P^2_\tau $, the complex projective space with the  conjugation action. Taking $ \alpha= \rho_{}=1+\sigma $
		\begin{myeq}\label{diag:cohoptwo}
			\xymatrix@R=0.5cm@C=0.5cm{\uH_G^{\rho_{}}(\C P^2_\tau;\Zt) \ar[d]^{\cong}\ar[rr]^-{\tilde{Sq^2}} 
				&&\uH_G^{\rho+2}(\C P^2_\tau;\Zt)\ar[d]^{\cong}
				\\ 
				\Zt & &\langle  \Lambda \rangle.
			}
		\end{myeq} 
		To see this recall from Proposition \ref{prop:cpnconj} that
		\begin{equation*}
			{\C P^2_\tau}\smas H\Z\simeq \bigvee_{i=1}^{2} S^{i+i\sigma}\smas H\uZ.
		\end{equation*}
		Hence  \begin{equation*}
			\tH^{\bigstar}_{C_2}(\C P^2_\tau;\Zt)\cong \bigoplus_{i=1}^{2} \tH_{C_2}^{\bigstar-i-i\sigma}(S^0;\Zt).
		\end{equation*}
		Applying \eqref{cp-comp z2 cof}, we obtain the required Mackey functors in \eqref{diag:cohoptwo}. Since $ Sq^2(x)=x^2 $ for a generator $ x\in H^2(\C P^\infty) $,  the diagram \eqref{diag:cohoptwo}  yields the following  commutative diagram  
		\begin{equation*}
			\xymatrix@R=0.1cm{  \Z/2 \ar@/_.5pc/[ddd]_{1}\ar[rr]^{ \tilde{Sq^2}} && 0 \ar@/_.5pc/[ddd]_{0}
				\\
				\\
				\\
				\Z/2\ar@/_.5pc/[uuu]_{  0  } \ar[rr]^{\cong} & & \Z/2 \ar@/_.5pc/[uuu]_{  0 },
			}
		\end{equation*}
		which is a contradiction.
	\end{exam}
	\bigskip
	Now we demonstrate the result in general.  Let $ {P(\mathcal{U})}^{\smas r} $ be the smash product of $ r $-copies of $ P(\mathcal{U}) $. Equation \eqref{eq:un4} gives us
	\begin{align*}
		F({P(\mathcal{U})}^{\smas r},H\Zp)&\simeq F_{H\uZ-\textup{mod}}({P(\mathcal{U})}^{\smas r}\smas H\uZ,H\Zp)\\
		&\simeq \bigwedg_{k_j=0}^{\infty} \bigwedg_{i_j=0}^{p-1} S^{-(k_1\phi_p+i_1 \lambda+\cdots+k_r\phi_p+i_r \lambda)}\smas H\Zp 
	\end{align*}
	where $ j\in \{1,\cdots, r\} $ and  $ (k_j,i_j)\ne (0,0) $. The last equivalence comes from the fact that $ {P(\mathcal{U})}^{\smas r}\smas H\uZ $ is a wedge of suspensions of $ H\uZ $ with finitely many $ \Sigma^V H\uZ $ of a given dimension $ V $.
	Hence 
	\begin{equation*}
		\tH^{\bigstar}_{C_p}({P(\mathcal{U})}^{\smas r};\Zp)\cong \bigoplus_{k_j=0}^{\infty}\bigoplus_{i_j=0}^{p-1} \tH_G^{\bigstar-(k_1\phi_p+i_1 \lambda+\cdots+k_r\phi_p+i_r \lambda)}(S^0;\Zp)
	\end{equation*}
	where $ j\in \{1,\cdots, r\} $ and  $ (k_j,i_j)\ne (0,0) $. We are now ready to prove Theorem \ref{thm:coh op}.
	
	\medskip
	\textit{Proof of the Theorem \ref{thm:coh op} when $ p $ is odd.} Let $ \theta $ be a  cohomology operation of degree $ r $ for  $ r $  even, such that $ \theta $ does  not involve the B\"{o}ckstein. Note that this condition implies $(p-1)\mid r$. Let $s= \frac{r}{p-1}$. Consider the element $ x_1\otimes\cdots \otimes x_s\in \tH^{2s}({\C P^{\infty}}^{\smas s};\Z/p) $, where each $ x_i $ is a generator of $ \tH^2(\C P^{\infty};\Z/p) $. By an argument analogous to \cite[Ch. 3, Theorem 2]{MoTa68}, we obtain $ \theta(x_1\otimes\cdots \otimes x_s)\ne 0 $.  Now suppose we have a map  
	\[
	\tilde{\theta}: H\Zp\to \Sigma^{r}H\Zp
	\] 
	such that $ \Omega(\tilde{\theta})=\theta $. This will lead to a map of Mackey functors 
	$$ \tilde{\theta}: \uH_G^\alpha(X;\Zp) \to \uH_G^{\alpha+r}(X;\Zp).
	$$ 
	Let us take $ X={P(\mathcal{U})}^{\smas s}$  and $ \alpha =s\lambda $. We observe that 
	\begin{myeq}\label{diag:cohop2}
		\xymatrix@R=0.5cm@C=0.5cm{
			\uH_G^{s\lambda}({P(\mathcal{U})}^{\smas s};\Zp) \ar[d]^{\cong}\ar[rr]^-{\tilde{\theta}} 
			&&\uH_G^{s\lambda+r}({P(\mathcal{U})}^{\smas s};\Zp)\ar[rr]^-{proj}\ar[d]^{\cong}  &&  \Zp^*
			\\ 
			\Zp &&{\Zp^*}^{\oplus t}\oplus {\langle  \Z/p \rangle}^{\oplus\ell} &&
		}
	\end{myeq}
	for some integer $ t\ge 1 $ and $ \ell\ge 0 $. To see this, let $ \alpha=s\lambda-  k_1\phi_p-i_1 \lambda-\cdots-k_s\phi_p-i_s \lambda$ and $ \tilde{\alpha}= s\lambda+r-  \tilde{k}_1\phi_p-\tilde{i}_1 \lambda-\cdots-\tilde{k}_s\phi_p-\tilde{i}_s \lambda$. The condition $ (k_j,i_j)\ne (0,0) $ implies $ |\alpha|=2s-2p(k_1+\cdots+k_s)-2(i_1+\cdots+i_s)\le 0 $. If $ |\alpha|<0 $, then the Mackey functor is zero by \ref{cp-comp zp cof}.  So the left hand side of diagram \ref{diag:cohop2} turns out to be $ \Zp $. However the Mackey functor $ \Zp $ can not appear in the right hand side as the condition
	$ |\tilde{\alpha}| =0$ forces $ |\tilde{\alpha}^{C_p}| $ to be $ >0 $.  This is because 
	$$
	|\tilde{\alpha}|=2s+r-2p(\tilde{k}_1+\cdots+\tilde{k}_s)-2(\tilde{i}_1+\cdots+\tilde{i}_s) =0
	$$ implies some $ i_j\ne 0 $. Now if 
	$ |\tilde{\alpha}^{C_p}|=r-2(\tilde{k}_1+\cdots+\tilde{k}_r)\le 0 $ then $|\tilde{\alpha}|\le 2s+(2-2p)(\tilde{k}_1+\cdots+\tilde{k}_r)-2(\tilde{i}_1+\cdots+\tilde{i}_r)<0 $. So $ |\tilde{\alpha}^{C_p}| >0$. 
	Since $ \theta(x_1 \otimes \cdots \otimes x_s)\ne 0$, the map $ \Zp(G/e)\to \Zp^*(G/e) $ is an isomorphism. As in the example \ref{ex:coh op}, this gives a contradiction. \qed\par
	
	\smallskip
	
	As before $ p=2 $ case require a different argument which we detail below. 
	
	\medskip
	
	\textit{Proof of the Theorem \ref{thm:coh op} when $ p =2$.} Let $ \theta $ be a  cohomology operation of degree $ 2r $, such that $ \theta $ does  not involve the B\"{o}ckstein.  By an argument analogous to \cite[Ch. 3, Theorem 2]{MoTa68}, we obtain $ \theta(x_1\otimes\cdots \otimes x_r)\ne 0 $ where each $ x_i $ is a generator of $ \tH^2(\C P^{\infty};\Z/p) $ and $ x_1\otimes\cdots \otimes x_r\in \tH^{2r}({\C P^{\infty}}^{\smas r};\Z/p) $.  Now suppose we have a map  
	\[
	\tilde{\theta}: H\Zt\to \Sigma^{2r}H\Zt
	\] 
	such that $ \Omega(\tilde{\theta})=\theta $. This will lead to a map of Mackey functors 
	$$ 
	\tilde{\theta}: \uH_G^\alpha(X;\Zt) \to \uH_G^{\alpha+2r}(X;\Zt).
	$$ 
	Let us take $ X={\C P^\infty_\tau}^{\smas r}$  and $ \alpha =r\rho=r+r\sigma $. With the help of Proposition \ref{prop:cpnconj} and \eqref{cp-comp z2 cof} we derive that
	\begin{equation*}
		\xymatrix@R=0.5cm@C=0.5cm{
			\uH_G^{r\rho}({\C P^\infty_\tau}^{\smas r};\Zt) \ar[d]^{\cong}\ar[rr]^-{\tilde{\theta}} 
			&&\uH_G^{r\rho+2r}({\C P^\infty_\tau}^{\smas r};\Zt)\ar[rr]^-{proj}\ar[d]^{\cong}  &&  \Zt^*
			\\ 
			\Zt &&{\Zt^*}^{}\oplus {\langle  \Z/2 \rangle}^{\oplus\ell} &&
		}
	\end{equation*}
	for some integer  $ \ell\ge 0 $.
	Since $ \theta(x_1 \otimes \cdots \otimes x_r)\ne 0$, the map $ \Zt(G/e)\to \Zt^*(G/e) $ is an isomorphism. As in the examples, this gives a contradiction. \qed
	
	\section{Slice tower of $ P(V)_+\smas H\Z $} \label{slice}
	The slice filtration in the equivariant stable homotopy category was introduced by Hill, Hopkins and Ravenel in their proof of the Kervaire invariant   one problem \cite{HHR} (see also \cite{Hill-sliceprimer}). The associated slice tower is an equivariant analogue of the Postnikov tower.   We use the regular slice filtration on equivariant spectra  (see \cite{Ullman-slice}), which differs from the original formulation   by a shift of one \cite[Proposition 3.1]{Ullman-slice}.
	\begin{defn}
		Let $ \tau _{\ge n} $ denote the localizing subcategory of genuine $ G $-spectra which is generated by $ G $-spectra of the form $ G_+\smas_H S^{k\rho_H} $, where $ \rho_H $ is the  (real) regular representation of $ H $ and $k|H|\ge n $.
	\end{defn}
	Let $ E $ be a $ G $-spectrum. Then $ E $ is said to be \textit{slice $ n $-connective} (written as $ E\ge n $) if $ E\in \tau_{\ge n} $, and $ E $ is said to be \textit{slice $ n $-coconnective} (written as $ E<n $) if 
	\[
	[S^{k\rho_H+r}, E]^H=0
	\]
	for all subgroup $ H\le G $ such that $ k|H|\ge n $ and  for all $ r\ge 0 $. We say $ E $ is an \textit{$ n $-slice} if $ n\le E\le n $.\par

	In \cite{HillYarnall}, the authors provide an alternative criterion for something being slice connective using the geometric fixed point functor.
	\begin{thm}\cite[Theorem 3.2]{HillYarnall}\label{thm:HY} The representation sphere $ S^V $ is in $ \tau_{\ge n} $ if and only if for all $ H\subset G $,
		\[
		\dim V^H\ge n/|H|.
		\]
		We note the following result from \cite[Proposition 2.23]{Hill-sliceprimer}.
		\begin{prop}\label{prop:Hill-criterion}
			If $ X $ is in $ \tau_{\ge 0} $ and $ Y $ is in $ \tau_{\ge n} $, then $ X\smas Y $ is in $ \tau_{\ge n} $.
		\end{prop}		
	\end{thm}
	There has been a large number of computations of slices for equivariant spectra. They are either carried out in the case of $ MU $ or its variants (\cite{HHR,HHR-realKthy,HillShiWangXu}) or for spectra of the form $ \Sigma^n H\uZ $ (\cite{HHR-slicess,Yarnall-slice,GuillouYarnall, GhoshS-slice,Slone}). We show that our cellular filtration actually yields the slice filtration for $ P(\mathcal{U})_+\smas H\uZ $ up to a suspension and in addition, the additive decomposition proves that the slice tower is degenerate in the sense that the maps possess sections.
	\begin{thm}\label{thm:slice}
		The slice towers of $ \Sigma^2 P(\mathcal{U})_+\smas H\uZ $ and $ \Sigma^2 P(n\rho)_+\smas H\uZ $ are degenerate and these spectra are a wedge of slices of the form $ S^V\smas H\uZ $.
	\end{thm}
	\begin{proof}
		Theorem \ref{thm main simple} allows us to write  \[ P(\mathcal{U})_+\smas H\uZ \simeq \bigwedg_{\ell=0}^\infty H\uZ\smas S^{\phi_\ell},\]
		where  $ \phi_\ell=\lam^{-l}(\lambda+\cdots+\lambda^{\ell-1}) $ and $ \phi_0=0 $. We claim that each of $ \Sigma^2S^{\phi_\ell}\smas H\uZ $ is a $ 2\ell+2 $-slice. Let $ H=C_m $, a subgroup of $ G $. We verify
		\begin{equation*}\label{eq:umn2}
			[S^{k\rho_H+r}, S^{res_H(\phi_\ell)+2}\smas H\uZ]^H\cong
			H^{\alpha}_H(S^0;\uZ)=0
		\end{equation*}
		for  $ k|H|>2\ell+2 $ and  $ \alpha= \res_H(\phi_\ell)+2-r-k\rho_H$. We may write $ \ell  =qm +s$, where $ s<m $, and thus $ \res_H(\phi_\ell)=2q\rho_H+\lambda +\cdots +\lambda^s$. Since $ km>2\ell+2 = 2qm+2s+2$, we get $ k>2q $. Let $ k=2q+j $, so that
		$$\alpha= -r-(\lambda^{s+1}+\cdots+\lambda^{m-1})-(j-1)\rho_H.$$
		Therefore either by  Proposition \ref{thm:neg htpy gr}, or using  the fact that all the fixed point dimensions of $ \alpha $ are negative, we have  $ H^{\alpha}_H(S^0;\uZ)=0 $.\par
		To show $ S^{\phi_\ell+2}\smas H\uZ \ge 2\ell+2 $, using Proposition \ref{prop:Hill-criterion}, it is enough to prove $ S^{\phi_\ell+2}\in \tau_{\ge 2\ell+2}$ as $H\uZ $ is a $ 0 $-slice \cite[Proposition 4.50]{HHR}. Appealing to Theorem \ref{thm:HY},  $ S^{\phi_\ell+2}\in \tau_{\ge 2\ell+2}$ as
		\[
		\frac{2\ell+2}{m}=2q+\frac{2s+2}{m}\le 2q+2=\dim( S^{\phi_\ell+2})^H \quad \textup{~for all~} H\subset G.
		\]
	\end{proof}
	\begin{thm}\label{thm:slicequa}
		The slice towers of $ \Sigma^4 P(\mathcal{U}_\bH)_+\smas H\uZ $ and $ \Sigma^4 P(n\rho_\bH)_+\smas H\uZ $ are degenerate and these spectra are a wedge of slices of the form $ S^V\smas H\uZ $.
	\end{thm}
	\begin{proof}
		Theorem \ref{quatadd} allows us to write  \[ P(\mathcal{U}_\bH)_+\smas H\uZ \simeq \bigwedg_{\ell=0}^\infty H\uZ\smas S^{\xi_\ell},\]
		where  $ \xi_\ell=\psi^1+\cdots+\psi^{\ell} $ and $ \xi_0=0 $. 
		Let $ H=C_m $, a subgroup of $ G $. We may write $ \ell  =qm +s$, where $ s<m $. Consequently, $ \res_H(\xi_\ell)=4q\rho_H+\lambda +\cdots +\lambda^s$. Proceeding  as in the case of Theorem \ref{thm:slice}, we  deduce that each of $ \Sigma^4S^{\xi_\ell}\smas H\uZ $ is a $ 4\ell+4 $-slice.
	\end{proof}
	
	\section{Ring Structure of $ B_{G}S^1$}\label{cohring}
	The objective of this section is to compute the equivariant cohomology ring $ H^\bigstar_G(B_GS^1) $ for $ G=C_{p^m} $, $ p $ prime and $ m\ge 1 $. We use the identification $ B_GS^1\simeq P(\mathcal{U}) $ where $ \mathcal{ U} $ is a complete $ G $-universe. The additive decomposition of \S \ref{sec:additive} already provides a basis for the cohomology as a module over $ H^\bigstar_G(\pt)=\pi_{-\bigstar}^G(H\uZ) $. 
	\begin{mysubsection}{Multiplicative generators of $H^{\bigstar}_{G}(P(\mathcal{U}))$} The fixed point space $ P(\mathcal{U})^G $ is a disjoint union of $ n $-copies of $ \C P^\infty $ which are included in $ P(\mathcal{U}) $ as $ P(\infty \lambda^i)=\textup{colim}_k P(k\lambda^i) $ for $ 0\le i\le n-1 $.
		Let $q_i :P(\infty \lambda^i) \rightarrow P(\mathcal{U})$ denote the inclusion. 
		In particular, we have 
		\begin{myeq}\label{eq:q0}
			q_0 :P(\mathbb{C}^{\infty})\rightarrow P(\sU).
		\end{myeq}
		For the trivial $ G $-action on $\mathbb{C}P^{\infty}$, we get
		$$H^{\bigstar}_{G}(\mathbb{C}P^{\infty})= H^{*}(\mathbb{C}P^{\infty})\otimes H^{\bigstar}_{G}(\pt) =H^{\bigstar}_{G}(\pt)[x],$$
		where $x$ is the multiplicative  generator of $H^{2}(\C P^{\infty};\Z)$.

		Recall	from \S \ref{cons:alpha phi d} that the classes $ \alpha_{k\phi_n+\phi_i} $ form an additive generating set for $H^{\bigstar}_{G}(P(\mathcal{U}))$. The following notes the multiplicative generating set.
		\begin{prop}\label{prop:multi gene}
			The collection   $\{\alpha_{{\phi_d} }\mid d \mbox{ divides }n\}$  generates $H^{\bigstar}_{G}(P(\mathcal{U}))$ as an algebra over $H^{\bigstar}_{G}(\pt)$.
		\end{prop}
		\begin{proof}
			We show by induction that each $\alpha_{k\phi_n+\phi_i}  $ may be expressed as a sum of monomials on the $ \alpha_{\phi_d} $. Suppose we know that all generators with degree lower than  $ \alpha_{k\phi_n+\phi_i} $ can be written in terms of $  \alpha_{{\phi_d} }$'s.
			Let $ i=\sum_{j=0}^\ell r_jp^j $ where $ 0\le r_j\le p-1 $. The class $ \alpha_{\phi_n}^k  \alpha_{\phi_{p^\ell}}^{r_\ell}\cdots \alpha_{\phi_1}^{r_0} $ also belongs to $ H^{k\phi_n+\phi_i}_{G}(P(\sU)) $.
			So we may express the class as
			\begin{myeq}\label{eq:monomial}
				\alpha_{\phi_n}^k  \alpha_{\phi_{p^\ell}}^{r_\ell}\cdots \alpha_{\phi_1}^{r_0} =	c_{\phi_1}\alpha_{\phi_1}+\cdots +c_{k\phi_n+\phi_i}  \alpha_{k\phi_n+\phi_i}, \qquad \textup{~where~} c_{\phi_j}\in H^{\bigstar}_{G}(\pt). 
			\end{myeq}
			Note that in \eqref{eq:monomial},  a generator  $ \alpha_{t\phi_n+\phi_s} $ with degree  greater than the degree of $ \alpha_{k\phi_n+\phi_i} $ cannot appear; this is because  for  $ \zeta=k\phi_n+\phi_i-(t\phi_n+\phi_s) $, the group $  H^{\zeta}_{G}(\pt) =0$ by \eqref{grzero} as all the fixed point dimensions of $ \zeta  $ are negative. Now, except the class $ c_{k\phi_n+\phi_i} $, the degree of the other $ c_\phi $  must be greater than zero, hence their restriction to $ G/e $ is zero. Thus,  $ \res^G_e(\alpha_{\phi_n}^k  \alpha_{\phi_{p^\ell}}^{r_\ell}\cdots \alpha_{\phi_1}^{r_0})=x^{n+i} =\res^G_e(\alpha_{k\phi_n+\phi_i})$  by \eqref{eq:res g/e}. Therefore, $ c_{k\phi_n+\phi_i} $ must be $ 1 $.
		\end{proof}
		\begin{exam}
			Let $ G=C_p $. 	The two classes $ \alpha_{\phi_1}$ and $ \alpha_{\phi_p} $  generate  $H^{\bigstar}_{C_p}(P(\mathcal{U})) $.  The computations of Lewis \cite[\S 5]{Lewis} may be adapted to prove
			\begin{myeq}\label{eq:cp case ring str}
				H^{\bigstar}_{G}(P(\mathcal{U}))\cong H^{\bigstar}_{G}(\pt)[\alpha_{\phi_1}, \alpha_{\phi_p}]/(u_\lambda\alpha_{\phi_p}-\alpha_{\phi_1}\prod_{i=1}^{p-1}(ia_\lambda+\alpha_{\phi_1})).
			\end{myeq}
			The relation is obtained by restriction to various fixed points.
		\end{exam}
		\medskip
		
		The proof of Proposition \ref{prop:multi gene} also demonstrates that one may change the basis of $ H^\bigstar_G(P(\mathcal{U})) $ over $ H^\bigstar_G(\pt) $ from $ \{\alpha_{k\phi_n+\phi_i}\} $ to $ \{\alpha_{\phi_n}^k  \alpha_{\phi_{p^\ell}}^{r_\ell}\cdots \alpha_{\phi_1}^{r_0}\}$ where $ 0\le r_j\le p-1 $. Therefore, there exists a relation of the form 
		\[
		\alpha_{\phi_{p^j}}^p=c\alpha_{\phi_{p^{j+1}}}+\textup{~lower order terms}.
		\]
		By restricting to $ G/e $, we see that $ c $ must be $ u_{\lambda^{p^j}-\lambda^{p^{j+1}}} $. The lower order terms will be calculated by restriction to fixed points.
		The next result describes $ q_0^*(\alpha_{\phi_1}) $. 
		\begin{prop}\label{prop:q0 of alp lamb}
			$ q_0^*(\alpha_{\phi_1}) =u_\lambda x$.
		
		\end{prop}
		\begin{proof}
			At $ \deg  \lambda$, $ H^{\bigstar}_{G}(\mathbb{C}P^{\infty}) $ has a basis given by $ a_\lambda $ and $ u_\lambda x $.  So  $  q_0^*(\alpha_{\phi_1})= c_1 a_\lambda+c_2 u_\lambda x$. In the notation of \eqref{eq:phi d}, $ W_1= 1_{\C}+\lambda$. Consider the map $ i:P(W_1)=S^{\lambda^{-1}}\hookrightarrow  P(\mathcal{U})$ and $ f:\pt\hookrightarrow  \C P^\infty $. Consider the commutative diagram  
			\begin{center}
				\begin{tikzcd}
					H^2(\C P^\infty)		\ar[d,"\cong"]&	H^\lambda_G(P(\mathcal{U})) \arrow[r, "i^*"] \ar[l,"\res^G_e"']\arrow[d, "q_0^*"]& H^\lambda_G(S^{\lambda^{-1}}) \arrow[d, "q_0^*"] \\
					H^2(\C P^\infty)	& H_G^\lambda(\C P^\infty)\arrow[r, "f^*"]           \ar[l,"\res^G_e"']          & H_G^\lambda(\pt).            
				\end{tikzcd}
			\end{center}
			By \eqref{eq:res g/e},  $ \res^{G}_e(\alpha_{\phi_1})=x $. So the left commutative  square implies $ c_2 $ must be $ 1 $.   Next, observe that the map $ i^* $
			sends $\alpha_{\phi_1} $ to the generator corresponding to $ 1\in H^0_G(S^0)\simeq \tH^\lambda_G(S^\lambda)\subseteq H^\lambda_G(S^\lambda) $.
			The cofibre sequence   
			$ P(1_{\C})_+\simeq S^0 \xrightarrow{q_0} P(W_1)_+\to S^\lambda $ implies $ q_0^*i^*(\alpha_{\phi_1})=0 $. So $ c_1=0 $, and thus $ q_0^*(\alpha_{\phi_1}) =u_\lambda x$.
		\end{proof} 
	\end{mysubsection}
	\begin{mysubsection}{Restrictions to fixed points}
		We adapt the approach of Lewis \cite{Lewis} to our case for calculating $ q_0^*(\alpha_{\phi_d}) $.	For  a subset  $ I $ of $ \underline{d-1}:=\{1,2,\cdots, d-1\} $, denote 
		\[
		\omega_{I}=\lambda^{-d}(1_\C+\Sigma_{i\in I}\lambda^i),
		\]
		and 
		\[
		V_{I,k}=1_\C+\Sigma_{i\in I}\lambda^i+\lambda^d+k\cdot1_\C,
		\]
		for $ k\ge 0 $.  Consider the cofibre sequence 
		\[
		P(V_{I,0}-\lambda^d)_+\to P(V_{I,0})_+\xrightarrow{\chi} S^{\omega_I},
		\] 
		which implies the long exact sequence  
		\[
		\cdots \to \tH^{\omega_I -1}_{G}(P(V_{I,0}-\lambda^d)_+)\ra \tH^{\omega_I }_{G}(S^{\omega_I })\xra{\chi^*} \tH^{\omega_I }_{G}(P(V_{I,0})_+)\ra \tH^{\omega_I }_{G}(P(V_{I,0}-\lambda^d)_+)\ra\cdots
		\]
		Define $ \Delta_{\omega _I}^{V_{I,0}}\in  H^{\omega_I }_{G}(P(V_{I,0}))$ to be $ \chi^*(1) $. Next we lift the class $ \Delta_{\omega _I}^{V_{I,0}} $ uniquely to the class  $ \Delta_{\omega _I}^{V_{I,k}}\in H^{\omega_I }_{G}(P(V_{I,k})) $ with the help of the
		cofibre sequence $ P(V_{I,\ell})_+\xrightarrow{\theta_\ell} P(V_{I,\ell+1})_+\to S^{V_{I,\ell}}.  $
		At degree $  \omega _I $ we get
		\[
			\cdots\ra \tH^{\omega _I}_G(S^{V_{I,\ell}})\ra \tH^{\omega _I }_G(P(V_{I,\ell+1})_+)\xra{\theta_\ell^*}\tH^{\omega _I }_G(P(V_{I,\ell})_+)\ra \tH^{\omega _I+1 }_{G}(S^{V_{I,\ell}})\to\cdots
		\]
		We claim $ \theta_\ell^* $ is an isomorphism. To see this we observe that as $ i<d $ and $ d $ is a power of $ p $, all the fixed-point dimensions of  $ \lambda^{i-d} $ and $ \lambda^i $ are same. Hence all the fixed-point dimensions of $ \omega_I-V_{I,\ell} $ are $ \le -2 $, so  $ H^{\omega _I-V_{I,\ell}}_G(\pt) =0$ and $ H^{\omega _I+1-V_{I,\ell} }_{G}(\pt)=0 $ for $ \ell\ge0  $.\par
		As the restriction of $ \chi^* $ to the orbit $ G/e $ is an isomorphism, we have 
		\begin{myeq}\label{eq:res omeg I cp}
			\res_e^G(\Delta_{\omega _I}^{V_{I,k}}) =x^{|I|+1}.
		\end{myeq}
		In the same spirit, using the cofibre sequence 
		\[
		P(V_{I,\ell})_+\hookrightarrow  P(V_{I,\ell+1})_+\xrightarrow{\chi} S^{V_{I,\ell}},
		\]
		we may define the class $ \Omega^{V_{I,\ell+1}}_{V_{I,\ell} }\in H^{V_{I,\ell}  }_{G}(P(V_{I,\ell+1}))  $ to be the image of $ \chi^*(1) $ where $ 1\in \tH^{V_{I,\ell} }_{G}(S^{V_{I,\ell} })\cong \Z $. As for $\Delta^{V_{I,k}}_{\omega_I}$, this lifts uniquely to define  the class  $ \Omega^{V_{I,k}}_{V_{I,\ell} }\in H^{V_{I,\ell}  }_{G}(P(V_{I,k}))  $. We define $\Omega^{V_{I,k}}_{V_{I,k}}=0$. Its restriction to the orbit $ G/e $ is 
		\begin{equation*}\label{}
			\res_e^G(\Omega^{V_{I,k}}_{V_{I,\ell} }) =x^{|I|+\ell+2}.
		\end{equation*}
	\end{mysubsection}
	For $ i\in I $, let $ \tau _{i,k} $ (or simply $ \tau _i $)  be the inclusion map $ P(V_{I\setminus\{i\},k}) \hookrightarrow  P(V_{I,k})$.   
	\begin{prop}\label{prop:cp tau ome1} For the map $ \tau _{i,k} $ we get the following 
		\begin{enumerate}
			\item$ \tau_{i,k}^*(\Delta^{V_{I,k}}_{\omega _I})=\Theta_{i,d}\cdot a_{\lambda^i} \Delta_{\omega_{I\setminus\{ i\}} }^{V_{I\setminus \{i\},k}}+u_{\lambda^i}\Omega_{V_{I\setminus\{ i\},0} }^{V_{I\setminus \{i\},k}}. $
			\item $ \tau_{i,k}^*(\Omega^{V_{I,k}}_{V_{I,\ell} })=a_{\lambda^i} \Omega_{V_{I\setminus\{i\},\ell} }^{V_{I\setminus \{i\},k}}+u_{\lambda^i}\Omega_{V_{I\setminus\{ i\},\ell+1} }^{V_{I\setminus \{i\},k}} $ \quad \textup{for $ 0\le \ell< k $.}
		\end{enumerate}
	\end{prop}
	\begin{proof}
		We prove $ (1) $. The proof of $ (2) $ is analogous.	We start with the representation $ V_{I,0} $, and successively add $ 1_\C $ to reach $ V_{I,k} $.  First consider the  following diagram. 
		\begin{myeq}\label{diag:step1}
			\begin{tikzcd}
				P(V_{I\setminus\{ i\},0}-\lambda^d)_+ \ar[r, hook, ""] \arrow[d,hook] & P(V_{I,0}-\lambda^d)_+\ar[d,hook, ""] \\
				P(V_{I\setminus\{ i\},0})_+ \ar[r, hook, "\tau_{i,0}"] \ar[d] & P(V_{I,0})_+\arrow[d, ""] \\
				S^{\omega _{I\setminus\{ i\}}} \ar[r, "a_{\lambda^{i-d}}"']                       & S^{\omega _I}          
			\end{tikzcd}  
		\end{myeq}	
		Since $ i<d $, the $ p $-adic valuation of $ i-d $ is same as $ i $. We have
		\begin{myeq}\label{eq:a lambda idntific}
			a_{\lambda^{i-d}} = \Theta_{i,d}\cdot a_{\lambda^i}
		\end{myeq}
		by Proposition \ref{prop:rel betw a lambda}. At $ \deg \omega _I $, the bottom commutative square gives us 
		\[
		\tau _{i,0}^*(\Delta_{\omega _I}^{V_{I,0}})=\Theta_{i,d}\cdot a_{\lambda^i}\Delta_{\omega_{I\setminus\{ i\}} }^{V_{I\setminus \{i\},0}}.
		\]
		In the next step, we add $ 1_\C $ to the representations in the middle row  of the  diagram \eqref{diag:step1}. This fits into the following diagram.
		\begin{myeq}\label{diag:step2}
			\begin{tikzcd}
				P(V_{I\setminus\{ i\},0})_+ \ar[r, hook, "\tau_{i,0}"] \arrow[d,hook,"+1_\C"'] & P(V_{I,0})_+\ar[d,hook, "+1_\C"] \\
				P(V_{I\setminus\{ i\},1})_+ \ar[r, hook, "\tau_{i,1}"] \ar[d] & P(V_{I,1})_+\arrow[d, ""] \\
				S^{V_{I\setminus\{ i\},0}} \ar[r, ""']                       & S^{V_{I,0}}          
			\end{tikzcd}  
		\end{myeq}	
		Since we have built $ P(V_{I\setminus\{ i\},0}) $ by attaching cells in a proper order, the boundary maps in the cohomology long exact sequence induced by the left-hand cofibration sequence are trivial. Moreover, $ H^\bigstar_G(P(V_{I\setminus\{ i\},0}) ) $ is free as $ H^\bigstar_G(\pt) $-module. So 
		\[
		H^{\omega _I}_G(P(V_{I\setminus\{ i\},1})\cong H^{\omega _I}_G(P(V_{I\setminus\{ i\},0}))\oplus H^{\omega _I-{V_{I\setminus \{i\},0}}}_G(\pt). 
		\] 
		Further, $ H^{\omega _I-{V_{I\setminus \{i\},0}}}_G(\pt) \cong  \Z$ generated by the class $u_{\lambda^i}$.  So  the diagram \eqref{diag:step2} implies
		\[
		\tau_{i,1}^*(\Delta^{V_{I,1}}_{\omega _I})=\Theta_{i,d}\cdot a_{\lambda^i} \Delta_{\omega_{I\setminus\{ i\}} }^{V_{I\setminus \{i\},1}}+c \cdot u_{\lambda^i}\Omega_{V_{I\setminus\{ i\},0} }^{V_{I\setminus \{i\},1}}
		\]
		for some $ c\in \Z $. We claim that $ c=1$. This is because from the next step onwards (where in each step we add a copy of $ 1_\C $) we have 
		$ H^{\omega _I}_G(P(V_{I\setminus\{ i\},k}))\cong H^{\omega _I}_G(P(V_{I\setminus\{ i\},1})) $ as $ H^{\omega _I-{V_{I,\ell}}}_G(\pt )=0 $ for $ \ell\ge 1 $. As a consequence, 
		\[
		\tau_{i,k}^*(\Delta^{V_{I,k}}_{\omega _I})=\Theta_{i,d}\cdot a_{\lambda^i} \Delta_{\omega_{I\setminus\{ i\}} }^{V_{I\setminus \{i\},k}}+c \cdot u_{\lambda^i}\Omega_{V_{I\setminus\{ i\},0} }^{V_{I\setminus \{i\},k}}.
		\]
		The map $ \tau _{i,k}^* $ at the orbit $ G/e $ is an isomorphism. Moreover,  by \eqref{eq:res omeg I cp} $ \res_e^G(\Delta_{\omega _I}^{V_{I,k}}) =x^{|I|+1}=\res_e^G(\Omega_{V_{I\setminus\{ i\},0} }^{V_{I\setminus \{i\},k}}). $ So $ c $ must be $ 1$, otherwise, we get  a contradiction by restricting to the orbit $ G/e $. This completes the proof of $ (1) $. 
	\end{proof}
	In case $ d=p^m $, then the following simplification occurs as $\frac{p^m}{i}a_{\lambda^i}=0$ \eqref{eq:a rel Z}. 
	\begin{prop}\label{prop:cp tau ome3} For the map $ \tau _{i,k} $ we get the following 
		\[
		\tau_{i,k}^*(\Delta^{V_{I,k}}_{\omega _I})=a_{\lambda^i} \Delta_{\omega_{I\setminus\{ i\}} }^{V_{I\setminus \{i\},k}}+u_{\lambda^i}\Omega_{V_{I\setminus\{ i\},0} }^{V_{I\setminus \{i\},k}}. 
		\]
	\end{prop}
	If we work with  $ \Zp $-coefficients, then $ a_{\lambda^{i-d}} =  a_{\lambda^i} $ as  $ \Theta_{i,d}\equiv 1 \pmod p$ by Proposition \eqref{prop:rel betw a lambda}. This simplification leads us to  
	\begin{prop}\label{prop:cp tau ome2} In $ \Zp $-coefficient, for the map $ \tau _{i,k} $ we get the following 
		\[
		\tau_{i,k}^*(\Delta^{V_{I,k}}_{\omega _I})= a_{\lambda^i} \Delta_{\omega_{I\setminus\{ i\}} }^{V_{I\setminus \{i\},k}}+u_{\lambda^i}\Omega_{V_{I\setminus\{ i\},0} }^{V_{I\setminus \{i\},k}}. 
		\]
		The value of $ \tau_{i,k}^*(\Omega^{V_{I,k}}_{V_{I,\ell} })$ remains same as in Proposition \ref{prop:cp tau ome1}.
	\end{prop}
	\begin{rmk}\label{rmk:cp empty case}
		Note that, for the group $ C_p $, we have $ P({V_{\varnothing,k}})=P(\C^{k+2})$. Hence the class $ \Delta^{V_{\varnothing,k}}_{\omega _\varnothing} $ is the class $ x\in H^{2}_G(P(\C^{k+2}))  $, and the class $ \Omega^{V_{\varnothing,k}}_{V_{\varnothing,\ell}}$ is the class $ x^{\ell+2} \in H^{2\ell+4}_G(P(\C^{k+2})) $.
	\end{rmk}
	\begin{prop}\label{prop:empty case}
		In the case when $ I =\emptyset $, we obtain
		\begin{enumerate}
			\item 	$ \tau_{d,k}^*(\Delta^{V_{\emptyset,k}}_{\omega _\emptyset})=u_{\lambda^d}\cdot x. $
			\item $ 	\tau_{d,k}^*(\Omega^{V_{\emptyset,k}}_{V _{\emptyset,\ell}})=a_{\lambda^d} \cdot x^{\ell+1}+u_{\lambda^d}\cdot x^{\ell+2}. $
		\end{enumerate}
	\end{prop}
	\begin{proof}
		Recall that $ V_{\emptyset,k}=1_\C+\lambda^d+k\cdot1_\C $. The cofibre sequence 
		\[
		P(1_\C)_+\xrightarrow{\tau _{d,0}}P(1_\C+\lambda^d)_+\to S^{\omega _\emptyset}
		\]
		implies $ \tau _{d,0} ^*(\Delta_{\omega _\emptyset}^{V_{\emptyset,0}})=0$. The rest of the proof is quite similar to  Proposition  \ref{prop:cp tau ome1}. So we describe it briefly. In the next step, we have
		\begin{equation*}\label{}
			\begin{tikzcd}
				P(1_\C)_+ \ar[r, hook, "\tau_{d,0}"] \arrow[d,hook,""'] & P(V_{\emptyset,0})_+\ar[d,hook, ""] \\
				P(2\cdot1_\C)_+ \ar[r, hook, "\tau_{d,1}"] \ar[d] & P(V_{\emptyset,1})_+\arrow[d, ""] \\
				S^2 \ar[r, ""']                       & S^{V_{\emptyset,0}}.          
			\end{tikzcd}  
		\end{equation*}	
		At degree $ \omega _\emptyset $, $ H_G^{\omega _\emptyset-2}(\pt)\cong \Z\{u_{\lambda^d}\} $, and $ H_G^{\omega _\emptyset-V_{\emptyset,0}}(\pt)=0$. Using restriction to the orbit $ G/e $, we may conclude that   $\tau_{d,1}^*(\Delta^{V_{\emptyset,k}}_{\omega _\emptyset})=\tau_{d,k}^*(\Delta^{V_{\emptyset,k}}_{\omega _\emptyset})=u_{\lambda^d}\cdot x  $.
	\end{proof}
	The following is a direct calculation
	\begin{prop}\label{prop:calcu}
		Let $ \mathcal{I}=\{i_1,\cdots, i_r\} $. Then	
		$$
		\tau_d \tau _{i_1}\cdots\tau _{i_r}(\Omega^{V_{I,k}}_{V_{I,t}})=x^{t+1}(a_{\lambda^d}+u_{\lambda^d}\cdot x)\prod_{s=1}^r(a_{\lambda^{i_s}}+xu_{\lambda^{i_s}}).
		$$
	\end{prop}
	\begin{proof}   Let $ \mathcal{I}=\{i_1,\cdots, i_r\} $. Proposition \ref{prop:cp tau ome1} and Proposition \ref{prop:empty case} gives us
		\begin{equation*}	\tau _d \tau _{i_1}\cdots\tau _{i_r}(\Omega^{V_{I,k}}_{V_{I,t}})= \tau_d\Big[
			\sum_{\ell=t}^{t+r}
			\big(\Omega_{V_{\emptyset,\ell}}^{V_{\emptyset,k}}
			\sum_{{\{j_1,\cdots,j_{\ell-t}\}\subseteq \mathcal{I}}}
			u_{\lambda^{j_1}}\cdots u_{\lambda^{j_{\ell-t}}}
			a_{\lambda^{j_{\ell-t+1}}}
			\cdots a_{\lambda^{j_r}}\big)
			\Big].
		\end{equation*}
		Further applying $ \tau _d $ to $ \Omega_{V_{\emptyset,\ell}}^{V_{\emptyset,k}} $, and  taking  $ x^{t+1} $ common, we simplify the right hand side as
		\[x^{t+1}(a_{\lambda^d}+u_{\lambda^d}\cdot x)
		\sum_{\ell=0}^{r} x^\ell
		\big(
		\sum_{{\{j_1,\cdots,j_{\ell}\}\subseteq \mathcal{I}}}
		u_{\lambda^{j_1}}\cdots u_{\lambda^{j_{\ell}}}
		a_{\lambda^{j_{\ell+1}}}
		\cdots a_{\lambda^{j_r}}\big).
		\]
This easily factorizes to imply the result.
	\end{proof}
	Now we are in a position to determine $ q_0^*(\alpha_{\phi_d}) $.
	\begin{prop}\label{prop:q0 of alp} 
		$  q_0^*(\a_{\phi_d})=\sum_{i=0}^{d-1}\Big[(\prod_{j=1}^{i}\Theta_{j,d} a_{\lambda^j}) u_{\lambda^{i+1}}x\prod_{s=i+2}^{d}(a_{\lambda^s}+u_{\lambda^s} x)\Big].$
	\end{prop}
	\begin{proof}
		Consider the map $i: P(V_{\underline{d-1},d})\rightarrow P(\mathcal{U}) $ given by inclusion. We claim $ i^*(\alpha_{\phi_d})=\Delta_{\omega _{\underline{d-1}}}^{V_{\underline{d-1},d}} $. The reason is as follows: we started with  the classes $\alpha_{\phi_d}^{W_d}  $ and $ \Delta_{\omega _{\underline{d-1}}}^{V_{\underline{d-1},0}}$ which were essentially same.  Then we  extended these classes through a chain of isomorphisms by successively adding representation $ \lambda^i$  (resp. $ 1_\C $) to define the class $ \alpha _{\phi_d} $ (resp. $ \Delta_{\omega _{\underline{d-1}}}^{V_{\underline{d-1},d}} $). Since in the end, we have $ P(V_{\underline{d-1},d})\hookrightarrow P(\mathcal{U}) $, so $ i^*(\alpha_{\phi_d})=\Delta_{\omega _{\underline{d-1}}}^{V_{\underline{d-1},d}} $.\par
		To determine $ q_0^*(\alpha_{\phi_d})$,  it is enough to work out $ q_0^*(\Delta_{\omega _{\underline{d-1}}}^{V_{\underline{d-1},d}} )$. For this, we successively remove all the nontrivial representations from   	$ V_{\underline{d-1},d} $.  Now   $ q_0^*(\Delta_{\omega _{\underline{d-1}}}^{V_{\underline{d-1},d}} )=\tau_d	 \cdots\tau _2\tau_{1}(\Delta_{\omega _{\underline{d-1}}}^{V_{\underline{d-1},d}} ) $. Applying  Proposition \ref{prop:cp tau ome2}, this becomes
		\begin{myeq}\label{eq:cp tau compose}
			\tau _d \cdots\tau _2(\Theta_{1,d}\cdot a_{\lambda} \Delta_{\omega_{{\underline{d-1}}\setminus\{ 1\}} }^{V_{\underline{d-1}\setminus \{1\},d}})+\tau _d \cdots\tau _2(u_{\lambda}\Omega_{V_{\underline{d-1}\setminus\{ 1\},0} }^{V_{\underline{d-1}\setminus \{1\},d}}).
		\end{myeq}  
		Let $ z_s=a_{\lambda^s}+u_{\lambda^s} x$.	The   second term can be simplified by Proposition \ref{prop:calcu} to
		$ 	u_\lambda x \prod_{s=2}^{d} (a_{\lambda^s}+u_{\lambda^s}x)=u_\lambda x \prod_{s=2}^{d} z_s.$
		Now applying $ \tau _2 $ in \eqref{eq:cp tau compose} and repeating the  above procedure,  we get
		\[
		\tau _d \cdots\tau _3(\Theta_{1,d}\cdot a_{\lambda}\Theta_{2,d}\cdot a_{\lambda^2} \Delta_{\omega_{{\underline{d-1}}\setminus\{ 1,2\}} }^{V_{\underline{d-1}\setminus \{1,2\},d}})+u_\lambda x \prod_{s=2}^{d} z_s+\Theta_{1,d}\cdot a_\lambda u_{\lambda^2} x \prod_{s=3}^{d} z_s.
		\]
		Repeating this process up to $ \tau_d $ we obtain the required.
	\end{proof}
	When $ d=p^m $, in the expression of  $ \tau_{i,k}^*(\Delta^{V_{I,k}}_{\omega _I}) $,  the numbers $ \Theta_{i,d} $ becomes $ 1 $   (cf. Proposition \ref{prop:cp tau ome3} and  \ref{prop:cp tau ome1}). Using Proposition \ref{prop:cp tau ome3} we obtain the following simplification.
	\begin{prop}\label{prop:q0 of alp zp d=pm}
		$ q_0^*(\a_{\phi_{p^m}})=\prod_{i=1}^{p^m}(a_{\lambda^i}+xu_{\lambda^i})- \prod_{i=1}^{p^m} a_{\lambda^i}	 $.
	\end{prop}
	\begin{proof}
	We use $\frac{p^m}{i}a_{\lambda^i}=0$ \eqref{eq:a rel Z} so that the terms $\Theta_{j,p^m}$ does not appear in this case as in the expression proved in Proposition \ref{prop:q0 of alp}. 
	Write $ z_s=a_{\lambda^s}+u_{\lambda^s} x$. We have 	
\[\begin{aligned}  
q_0^*(\a_{\phi_{p^m}})&=\sum_{i=0}^{p^m-1}\Big[(\prod_{j=1}^{i} a_{\lambda^j}) u_{\lambda^{i+1}}x\prod_{s=i+2}^{p^m}(a_{\lambda^s}+u_{\lambda^s} x)\Big] \\
&=\sum_{i=0}^{p^m-1}\Big[(\prod_{j=1}^{i} a_{\lambda^j}) u_{\lambda^{i+1}}x\prod_{s=i+2}^{p^m}z_s\Big] \\
&=\sum_{i=0}^{p^m-1}\Big[(\prod_{j=1}^{i} a_{\lambda^j}) \prod_{s=i+1}^{p^m}z_s - (\prod_{j=1}^{i+1} a_{\lambda^j}) \prod_{s=i+2}^{p^m}z_s\Big]\\
&= \prod_{i=1}^{p^m} z_s - \prod_{i=1}^{p^m} a_{\lambda^i}.
\end{aligned}\]
	\end{proof}	
	Either taken in $ \Zp $-coefficient or in case of $ d=p^m $, the expression of $ \tau_{i,k}^*(\Delta^{V_{I,k}}_{\omega _I}) $ is same 
	(cf. Proposition \ref{prop:cp tau ome3} and \ref{prop:cp tau ome2}). As a result we may proceed as in Proposition \ref{prop:q0 of alp zp d=pm} to obtain the following 
	\begin{prop}\label{prop:q0 of alp zp}
		With $ \Zp $-coefficients, we have  
		$$ q_0^*(\a_{\phi_{d}})=\prod_{i=1}^{d}(a_{\lambda^i}+xu_{\lambda^i})- \prod_{i=1}^{d} a_{\lambda^i}.
		$$
	\end{prop}
	Once we have the expressions for $q_0$ on the multiplicative generators, we relate them to obtain the relations in the cohomology ring. The following proposition states that $q_0^\ast$ is injective, which means that the image of $q_0^\ast$ may be used to detect relations. 
	
	\begin{prop}\label{prop:q0 injective}
	For every $j\geq 1$, the map $ q_0^* $ is injective at the degree $\zeta_{p^j}=\lambda+\cdots+\lambda^{p^j-1}+\lambda^{p^{j-1}}$.
	\end{prop}
	\begin{proof}
		Recall from  the additive decomposition of Theorem \ref{thm main simple} that 
		$$H^{\bigstar}_{G}(P(\mathcal{U}))\cong \bigoplus\limits_{i\geq0}H^{\bigstar -\phi_{i}}_{G}(\pt)\{\alpha_{\phi_i}\}.$$
		The notation here means that  $ \alpha_{\phi_i}$ generates  the factor  $ H^{\bigstar -\phi_{i}}_{G}(\pt) $ of the free $ H^{\bigstar}_{G}(\pt) $-module $ H^{\bigstar}_{G}(P(\mathcal{U})) $. Thus, $H^{\zeta_{p^j}}_{G}(P(\mathcal{U})) \cong \bigoplus\limits_{i= 0}^{p^j} H^{\zeta_{p^j}-\phi_{i}}_{G}(\pt)\{\alpha_{\phi_i}\}$. The higher degree generators ($ \alpha_{\phi_{p^j+1}} $ onwards) can not appear as by Proposition \ref{mfun0} the group $ H^{\zeta_{p^j}-\phi_{p^j+\ell}}_{G}(\pt) $ is zero for $ \ell\ge 1 $. For $i\leq p^j -1 $, we use Theorem \ref{Zhtpy} to note that the element $ a_{\zeta_{p^j}-\phi_i}:= a_{\lam^{i+1}}\cdots a_{\lam^{p^j-1}}a_{\lambda^{p^{j-1}}}$ generates the group  $ H^{\zeta_{p^j}-\phi_{i}}_{G}(\pt)\cong  \Z/p^{m-j+1}$ (look at the discussion following Theorem \ref{Zhtpy}).
		Therefore the element   $a_{\zeta_{p^j}-\phi_{i}}\a_{\phi_{i}}\in H^{\zeta_{p^j}}_{G}(P(\mathcal{U}))$
		also  has order $p^{m-j+1}$.
		Hence, 
		\[H^{\zeta_{p^j}}_{G}(P(\mathcal{U})\cong \Z\bigoplus \Big(\Z/p^{m-j+1}\Big)^{\oplus p^j}.\]
		 Also, we have 
		$$
		H^{\bigstar}_{G}(P(\infty 1_{\C}))\cong \bigoplus\limits_{i\geq0}H^{\bigstar - 2i}_{G}(\pt)\{x^{i}\}.
		$$
		So $H^{\zeta_{p^j}}_{G}(P(\infty1_{\C})) \cong \bigoplus\limits_{i= 0}^{p^j} H^{\zeta_{p^j}-2i}_{G}(\pt)\{x^{i}\}$. By the discussion following Theorem \ref{Zhtpy}, the group  $H^{\zeta_{p^j}-2i}_{G}(\pt) \cong  \Z/p^{t_{j,i}}$ for $i<p^j$, where  
		\[
		t_{j,i}=\max\big[m- \max\{v_p(s_1),\cdots, v_p(s_{p^j-i})\}  ~~|~~ s_1,\cdots, s_{p^{j}-i}\in \{1,\cdots, p^j-1,p^{j-1}\}\big].
		\]
		This tells us 
		\[ H^{\zeta_{p^j}}_{G}(P(\infty 1_{\C}))\cong \Z\bigoplus\Z/p^{t_{j,p^j-1}}\bigoplus\cdots\bigoplus\Z/p^{t_{j,0}} .\]
		We observe that the element 
		$$
		u_{\phi_{i}}a_{\zeta_{p^j}-\phi_{i}} := u_{\lam}u_{\lam^2}\cdots u_{\lam^i}a_{\lam^{i+1}}\cdots a_{\lam^{p^j-1}}a_{\lambda^{p^{j-1}}} \in H^{\zeta_{p^j}-2i}_{G}(\pt) 
		$$
		has order $p^{m-j+1}$.
		So the element $a_{\zeta_{p^j}-\phi_{i}}u_{\phi_{i}}x^{i}\in H^{\zeta_{p^j}}_{G}(P(\infty 1_{\C}))$ is also of order 
		$p^{m-j+1}$ in $\Z/t_{j,i}$. Since $ \res^G_e(\alpha_{\phi_i}) =x^{i}$, we have $\res^G_e(q_{0}^{*}(\a_{\phi_{i}}))=x^{i}$. This  implies  $q_{0}^{*}(\alpha_{\phi_{i}})= u_{\phi_{i}}x^{i} + \Sigma _{l=0}^{i-1}c_lx^{l}$, for some coefficients $ c_l $. Thus 
		$$q_{0}^{*}(a_{\zeta_{p^j}-\phi_{i}}\a_{\phi_{i}})= a_{\zeta_{p^j}-\phi_{i}}q_{0}^{*}(\a_{\phi_{i}}) = a_{\zeta_{p^j}-\phi_{i}}u_{\phi_{i}}x^{i} + a_{\zeta_{p^j}-\phi_{i}}\sum_{l=0}^{i-1}c_lx^{l}. $$
		Therefore $q_{0}^{*}$ as a map  
		$$\Z\bigoplus\big(\Z/p^{m-j+1}\big)^{\oplus p^j} \to \Z\bigoplus\Z/p^{t_{j,p^j-1}}\bigoplus\cdots\bigoplus\Z/p^{t_{j,0}}$$ 
		is a lower triangular matrix of the form
		\[
		\left(
		\begin{array}{ccccc}
			1        &                          &           &  \\
			*	         & k_{j,p^j-1} &           & \text{\huge 0}\\
			\vdots	& \vdots           &    \ddots       &                \\
			*	          &  *                 & \dots  & k_{j,0}
		\end{array}
		\right)
		\]  
		where $k_{j,i}=p^{t_{j,i} - m+j-1}$. Hence $q_{0}^{*}$ is injective at the degree $\zeta_{p^j}$.
	\end{proof}

\begin{mysubsection}{Relations for complex projective spaces}
	Note that if $ d=p^m $, then $ a_{\lambda^d}=0 $ and $ u_{\lambda^d}=1 $, so Proposition \ref{prop:q0 of alp zp d=pm} simplified to  	$ q_0^*(\a_{\phi_n})=\prod_{i=1}^{n-1}x(a_{\lambda^i}+xu_{\lambda^i})	 $. For the group $ C_p $, using Proposition \ref{prop:rel betw a lambda}, this further reduces  to   $ q_0^*(\a_{\phi_p})=x\prod_{i=1}^{p-1}(ia_{\lam}+ u_{\lam}x) 	 $. Using the fact that $ q_0^*(\alpha_{\phi_1})=u_\lambda x $ from \eqref{prop:q0 of alp lamb}, we see that 
	$$ q_0^*(u_\lambda\alpha_{\phi_p}-\alpha_{\phi_1}\prod_{i=1}^{p-1}(ia_\lambda+\alpha_{\phi_1})) =0.
	$$
	Moreover,  Proposition \ref{prop:q0 injective} tells us that $ q_0^* $ is injective. 
	So the relation we obtain for $ C_p $ is
	\[
	u_\lambda\alpha_{\phi_p}-\alpha_{\phi_1}\prod_{i=1}^{p-1}(ia_\lambda+\alpha_{\phi_1}).
	\]
	\smallskip
	
	For general $C_{p^m}$ of order $n=p^m$, there are $m$ relations of the form 
	\[ u_{\lambda^{p^{i-1}}-\lambda^{p^i}}\alpha_{\phi_{p^i}} = \alpha_{\phi_{p^{i-1}}}^p + \mbox{ lower order terms},\]
	for $1\leq i \leq m$. In fact, the proof of Proposition \ref{prop:multi gene} implies that the coefficients of the lower order terms are expressible as a sum of monomials with coefficients  that are linear combinations of products of $a_{\lambda^j}$. The naive idea is to apply $q_0^\ast$ to such an equation to determine all the coefficients. However,  the expression in Proposition \ref{prop:q0 of alp} does not directly yield a simple closed relation. We are able to obtain a simple expression after mapping to  $\Z/p$-coefficients. 
	
	The first observation when we look at $\Z/p$-coefficients is that $q_0^\ast$ is no longer injective. For, in the proof of Proposition \ref{prop:q0 injective}, the diagonal entries in the lower triangular matrix other than at the top corner, turns out to be $0\pmod{p}$. We use the formula for $q_0^\ast$ and that it is injective with $\Z$-coefficients. Let $\RR_d$ denote the algebra 
	\[ \Z[u_{\lambda^i}, a_{\lambda^j}, u_{\lambda^{p^{d-1}}-\lambda^{p^d}}]/I,\]
where $I$ is the ideal generated by the relations \eqref{eq:a rel Z}, those in Proposition \ref{prop:un5} and	
\[	  u_{\lambda^{p^{d-1}}-\lambda^{p^d}}u_{\lambda^{p^d}}= u_{\lambda^{p^{d-1}}}, ~ u_{\lambda^{p^{d-1}}-\lambda^{p^d}}a_{\lambda^{p^d}}= p a_{\lambda^{p^{d-1}}}, \]
which maps to $\pi_{-\bigstar} H\uZ=H^\bigstar_G(\pt)$. The algebra $\RR_d$ contains the classes $u_{\lambda^i}$ and $a_{\lambda^j}$ but they are not required to satisfy the relation \eqref{eq:au rel Z}.  	Form the algebraic $q_0^\ast$ map 
\[Q_0: \RR_d[\alpha_{\phi_{p^j}}\mid 0\leq j\leq m] \to \RR_d[x]\]
given by the formula in Proposition \ref{prop:q0 of alp}. In the absence of the relation \eqref{eq:au rel Z} in $\RR_d$, the lower triangular matrix in the proof of Proposition \ref{prop:q0 injective} gets replaced by one where the diagonal entries are inclusions of the corresponding summand. This becomes injective even after tensoring with $\Z/p$. The algebra $\RR_d[\alpha_{\phi_{p^j}}\mid 0\leq j\leq m]$ is denoted $\RR_d(P(\UU))$. 

 We thus work with the diagram 
\[ \xymatrix{ \RR_d(P(\UU)) \ar[d]^{Q_0} \ar[r]^-{\nu_p} & \Z/p\otimes \RR_d(P(\UU))\ar[d]^{Q_0}\\	
	 \RR_d[x]  \ar[r]^-{\nu_p} & \Z/p\otimes \RR_d[x], }\]
	 and seek relations $\chi$ which maps to $0$ in $\Z/p \otimes \RR_d[x]$. It follows that $\chi \equiv 0 \pmod{p}$ in $\RR_d(P(\UU))$ and thus gives a relation in $H^\bigstar_G(P(\UU);\uZp)$. We note from Proposition \ref{prop:rel betw a lambda}  that 
	 \begin{myeq} \label{aident} 
	 a_{\lambda^{kp^{r-1}+i}} = (1+kp^{r-1}\cdot i^{-1}) a_{\lambda^i}, \mbox{ and hence, } a_{\lambda^{kp^{r-1}+i}} \equiv a_{\lambda^i} \pmod{p}, \mbox{ for } i< p^{r-1}.
	 \end{myeq}
	   The following is a consequence of the identity $ \prod_{i=1}^{p-1}(x+i)\equiv x^{p-1}-1 \pmod p $.
	\begin{lemma}\label{numbthyres}
		$$
		\prod_{i=1}^{p-1}(ia_{\lambda^{p^{r-1}}}+xu_{\lambda^{p^{r-1}}})=(xu_{\lambda^{p^{r-1}}})^{p-1}-(a_{\lambda^{p^{r-1}}})^{p-1}.
		$$
	\end{lemma}
We write $P(z,w)=(z-w)^{p-1}-w^{p-1}$. Define the following notations
	 \begin{myeq}\label{btnot}
	 \begin{aligned} 
	 &\BB_r=\prod_{i=1}^{p^{r}-1}(a_{\lam^{i}}+u_{{\lam}^{i}}x) \in \RR_d[x] , \\ 
&\TT_r = \a_{\phi_{p^{r}}} +  \prod_{i=1}^{p^r} a_{\lambda^i} \in \RR_d(P(\UU)), \\
&\T_r=Q_0(\TT_r) = \BB_r \cdot (xu_{\lambda^{p^r}}+a_{\lambda^{p^r}}) \pmod{p}, \mbox{ by Proposition \ref{prop:q0 of alp zp}}, \\
& \A_0=P(\T_0,a_\lambda), \mbox{ and inductively, } \A_j= P(\T_j,a_{\lambda^{p^j}}\prod_{i=0}^{j-1} \A_i)\\
& \cA_0=P(\TT_0,a_\lambda), \mbox{ and, } \cA_j= P(\TT_j,a_{\lambda^{p^j}}\prod_{i=0}^{j-1} \cA_i), \mbox{ so that } Q_0(\cA_j)=\A_j.	 \end{aligned}
\end{myeq}
	We now have the following relation with $\Z/p$ coefficients.
	\[ \begin{aligned} 
	\BB_r & = \prod_{i=1}^{p^r -1} (a_{\lam^{i}}+u_{{\lam}^{i}}x) \\
	         & = \prod_{i=1, p^{r-1}\nmid i}^{p^r -1} (a_{\lam^{i}}+u_{{\lam}^{i}}x) \prod_{j=1}^{p-1} (a_{\lam^{jp^{r-1}}}+u_{{\lam}^{jp^{r-1}}}x) \\ 
	         &=\BB_{r-1}^p \prod_{j=1}^{p-1} (ja_{\lam^{p^{r-1}}}+u_{{\lam}^{p^{r-1}}}x) \mbox{ by \eqref{aident}, and Proposition \ref{prop:rel betw a lambda}} \\ 
	         &=\BB_{r-1}^p \big((xu_{\lambda^{p^{r-1}}})^{p-1}- a_{\lambda^{p^{r-1}}}^{p-1}\big)\\
	         &=\BB_{r-1}\big((\T_{r-1}- \BB_{r-1}a_{\lambda^{p^{r-1}}})^{p-1} - (\BB_{r-1}a_{\lambda^{p^{r-1}}})^{p-1}\big)\\
	         &= \BB_{r-1} P(\T_{r-1},a_{\lambda^{p^{r-1}}}\BB_{r-1}).  
	\end{aligned}\]
	From the expression, it inductively follows that 
	\begin{myeq}\label{Bexp}
	\BB_r= \prod_{i=0}^{r-1} \A_r.
	\end{myeq}
	We finally obtain
	\begin{prop}\label{r-reln} 
	With $\Z/p$-coefficients, the class $\a_{\phi_{p^{r}}}$ satisfies  the following relation
		\[
		u_{\lam^{p^{r-1}}-\lam^{p^{r}}}\a_{\phi_{p^{r}}}= \TT_{r-1}^p - a_{\lambda^{p^{r-1}}}^{p-1}\TT_{r-1}\big(\prod_{i=0}^{r-2} \cA_i\big)^{p-1}.
		\]
	\end{prop}
	   
	 \begin{proof}
	 As observed above, it suffices to prove with $\Z/p$-coefficients 
	 \[Q_0(u_{\lam^{p^{r-1}}-\lam^{p^{r}}}\a_{\phi_{p^{r}}})= \T_{r-1}^p - a_{\lambda^{p^{r-1}}}^{p-1}\T_{r-1}\big(\prod_{i=0}^{r-2} \A_i\big)^{p-1}= \T_{r-1}^p - a_{\lambda^{p^{r-1}}}^{p-1}\T_{r-1}\BB_{r-1}^{p-1}.\]
	We verify 
	 \[\begin{aligned}
	 Q_0(u_{\lam^{p^{r-1}}-\lam^{p^{r}}}\a_{\phi_{p^{r}}})&= u_{\lam^{p^{r-1}}-\lam^{p^{r}}} \big(\prod_{j=1}^{p^r}(xu_{\lambda^j}+a_{\lambda^j}) - \prod_{j=1}^{p^r} a_{\lambda^j} \big)\\
	 &= \BB_r x u_{\lambda^{p^{r-1}}} \\ 
	 &=\BB_{r-1}^p \big((xu_{\lambda^{p^{r-1}}})^p- a_{\lambda^{p^{r-1}}}^{p-1}x u_{\lambda^{p^{r-1}}}\big)\\
	 &=\T_{r-1}^p - \BB_{r-1}^p(a_{\lambda^{p^{r-1}}}^p + a_{\lambda^{p^{r-1}}}^{p-1}x u_{\lambda^{p^{r-1}}})\\
	 &=\T_{r-1}^p - a_{\lambda^{p^{r-1}}}^{p-1} \T_{r-1}\BB_{r-1}^{p-1}.
	 \end{aligned}\]
	 This completes the proof. 
	 \end{proof}  
	 
	 \smallskip 
	
	We now summarize the computation in the following theorem. 
	\begin{thm}\label{ringstrs1}
	The cohomology ring 
	\[ H^\bigstar_G (B_GS^1;\uZp) \cong H^\bigstar_G(\pt;\uZp)[\a_{\phi_0},\cdots,\a_{\phi_m}]/(\rho_1,\cdots, \rho_m).\]
	The relations $\rho_r$ are described by 
	\[\rho_r= u_{\lam^{p^{r-1}}-\lam^{p^{r}}}\a_{\phi_{p^{r}}}- \TT_{r-1}^p + a_{\lambda^{p^{r-1}}}^{p-1}\TT_{r-1}\big(\prod_{i=0}^{r-2} \cA_i\big)^{p-1}, \]
	where $\TT_j$ and $\cA_j$ are defined in \eqref{btnot}.
	\end{thm}
	\end{mysubsection}

	\begin{mysubsection}{Ring Structure of $B_{G}SU(2)$}
	As in the complex case, we get the multiplicative generators $\beta_{2\phi_d}$ of $H^{\bigstar}_{G}(P(\mathcal{U}_{\Bbb H}))$ at the degrees $2\phi_{d}$ for divisors $ d $ of $ n $. The construction is same as the previous construction of the class $\a_{\phi_{d}}$. Consider  the representation 
		$$W_d:=1_{\mathbb{H}} + \psi^{1} + \psi^{2}+\dots + \psi^{d}.
		$$
		We have the cofibre sequence 
		$$P(W_{d-1})_{+}\ra P(W_{d})_+\ra S^{\lam^{-d}\otimes_{\C}W_{d-1}}.
		$$
		At degree ${2\phi_{d}}$ the associated long exact sequence is 
		$$\cdots \ra \tH^{2\phi_{d}-1}(P(W_{d-1})_+)\ra \tH^{2\phi_{d}}(S^{2\phi_{d}})\ra \tH^{2\phi_{d}}(P(W_{d})_+)\ra \tH^{2\phi_{d}}(P(W_{d-1})_+)\ra  \cdots
		$$
		since $\lam^{-d}\otimes_\C \sum_{i=0}^{d-1}(\lam^{i}+\lam^{-i})= \sum_{i=0}^{d-1}2\lam^{i-d} = 2\phi_{d}$. We define the classes $\beta_{2\phi_{d}}$ to be the image of $1$ in $ \tH^{2\phi_{d}}(S^{2\phi_{d}}) \cong \Z$. By induction, we extend $\beta_{2\phi_{d}}$ to get the generator of $B_{G}SU(2)$ at degree $2\phi_{d}$. We use the notation  $ \mathcal{L}_{j} $ for $ (\beta_{2\phi_{p^{j}}} +  \prod_{i=1}^{p^j} (a_{\lambda^i})^2 ) $. As in the complex case with $\uZp$-coefficients we have 
		\begin{thm}\label{ringstr2} 
The cohomology ring 
\[ H^\bigstar_G (B_GS^3;\uZp) \cong H^\bigstar_G(\pt;\uZp)[\beta_{2\phi_0},\cdots,\beta_{2\phi_m}]/(\mu_1,\cdots, \mu_m).\]
	The relations $\mu_r$ are described by 
	\[\mu_r= (u_{\lam^{p^{r-1}}-\lam^{p^{r}}})^2\beta_{2\phi_{p^{r}}}- \LL_{r-1}^p + a_{\lambda^{p^{r-1}}}^{2(p-1)}\LL_{r-1}\big(\prod_{i=0}^{r-2} \CC_i\big)^{p-1}, \]
	where $\CC_i$ is inductively defined as $\CC_0=P(\LL_0,a^2_\lambda)$, and,  $\CC_j= P(\LL_j,a^2_{\lambda^{p^j}}\prod_{i=0}^{j-1} \CC_i)$.
		\end{thm}

	\end{mysubsection}

	We conclude this section with the ring structure computation of ${\C P^\infty_\tau}$.
	\begin{mysubsection}{Ring Structure for ${\C P^\infty_\tau}$} Recall  the cofibre sequence from Proposition  \ref{prop:cpnconj} 
		\[
		{	\C P^{n-1}_\tau}_+\hookrightarrow  {\C P^n_\tau}_+\xrightarrow{\chi} S^{n+n\sigma}.
		\]
		This implies the long exact sequence 
		\[
		\cdots \to \tH_{C_2}^{n+n\sigma}(S^{n+n\sigma})\xrightarrow{\chi^*} \tH_{C_2}^{n+n\sigma}({\C P^n_\tau}_+)\to  \tH_{C_2}^{n+n\sigma}({\C P^{n-1}_\tau}_+)\to \cdots
		\]
		Observe  that $ \chi^*(1) $ is nonzero where $ 1\in \tH_{C_2}^{n+n\sigma}(S^{n+n\sigma})\cong \Z $. Let $ \epsilon_{n+n\sigma}\in \tH_{C_2}^{n+n\sigma}({\C P^n_\tau}_+) $ be the element $ \chi^*(1) $.  As the restriction of $ \chi^* $ to the orbit $ C_2/e $ is an  isomorphism, we have $ \res^{C_2}_e(\epsilon_{n+n\sigma})=x^n \in \tH^{2n}(\C P^n)$. We claim $ H^\bigstar_{C_2}({\C P^\infty_\tau}_+)$ is  the polynomial ring  $H^\bigstar_{C_2}(\pt)[\epsilon_{1+\sigma}] $. This follows from the fact that $ \uH^{n+n\sigma}_{C_2}({\C P^\infty_\tau}_+)\cong \uZ $ and $ \res^{C_2}_e(\epsilon_{n+n\sigma})=\res^{C_2}_e(\epsilon^n_{n+n\sigma}) $. Hence $ \epsilon_{n+n\sigma}=\epsilon^n_{1+\sigma} $. Therefore, 
		\begin{thm}\label{cptau2}
		We have an isomorphism of cohomology rings
		\[H^\bigstar_{C_2}(\C P^\infty_\tau)\cong H^\bigstar_{C_2}(\pt)[\epsilon_{1+\sigma}].\]
		\end{thm}
	\end{mysubsection}


\begin{thebibliography}{10}

\bibitem{Ang21}
{\sc V.~Angeltveit}, {\em The {P}icard group of the category of
  {$C_n$}-equivariant stable homotopy theory}, 2021.

\bibitem{BasuDey}
{\sc S.~Basu and P.~Dey}, {\em Equivariant cohomology for cyclic groups}.

\bibitem{BasuDeyKarmakar-NYJM}
{\sc S.~Basu, P.~Dey, and A.~Karmakar}, {\em Equivariant homology
  decompositions for cyclic group actions on definite 4-manifolds}, New York J.
  Math., 28 (2022), pp.~1554--1580.

\bibitem{BG19}
{\sc S.~Basu and S.~Ghosh}, {\em Computations in {$C_{pq}$}-{B}redon
  cohomology}, Math. Z., 293 (2019), pp.~1443--1487.

\bibitem{BG20}
\leavevmode\vrule height 2pt depth -1.6pt width 23pt, {\em Equivariant
  cohomology for cyclic groups of square-free order}, 2020.
\newblock available at \url{https://arxiv.org/abs/2006.09669}.

\bibitem{BasuGhosh-cp}
\leavevmode\vrule height 2pt depth -1.6pt width 23pt, {\em Bredon cohomology of
  finite dimensional {$C_p$}-spaces}, Homology Homotopy Appl., 23 (2021),
  pp.~33--57.

\bibitem{Caruso}
{\sc J.~L. Caruso}, {\em Operations in equivariant {${\bf Z}/p$}-cohomology},
  Math. Proc. Cambridge Philos. Soc., 126 (1999), pp.~521--541.

\bibitem{Chonoles}
{\sc Z.~Chonoles}, {\em The {$RO(G)$}-graded cohomology of the equivariant
  classifying space {$B_G{SU(2})$}},  (2018).

\bibitem{Dre72}
{\sc A.~W.~M. Dress}, {\em Contributions to the theory of induced
  representations}, in Algebraic {$K$}-theory, {II}: ``{C}lassical'' algebraic
  {$K$}-theory and connections with arithmetic ({P}roc. {C}onf., {B}attelle
  {M}emorial {I}nst., {S}eattle, {W}ash., 1972), 1973, pp.~183--240. Lecture
  Notes in Math., Vol. 342.

\bibitem{Ferland-thesis}
{\sc K.~K. Ferland}, {\em On the {RO}({G})-graded equivariant ordinary
  cohomology of generalized {G}-cell complexes for {G} = {Z}/p}, ProQuest LLC,
  Ann Arbor, MI, 1999.
\newblock Thesis (Ph.D.)--Syracuse University.

\bibitem{FL04}
{\sc K.~K. Ferland and L.~G. Lewis, Jr.}, {\em The {$R{\rm O}(G)$}-graded
  equivariant ordinary homology of {$G$}-cell complexes with even-dimensional
  cells for {$G={\Bbb Z}/p$}}, Mem. Amer. Math. Soc., 167 (2004), pp.~viii+129.

\bibitem{GhoshS-slice}
{\sc S.~Ghosh}, {\em Slice filtration of certain {$C_{pq}$}-spectra}, New York
  J. Math., 28 (2022), pp.~1399--1418.

\bibitem{Greenlees}
{\sc J.~P.~C. Greenlees}, {\em Four approaches to cohomology theories with
  reality}, in An alpine bouquet of algebraic topology, vol.~708 of Contemp.
  Math., Amer. Math. Soc., [Providence], RI, [2018] \copyright 2018,
  pp.~139--156.

\bibitem{GreenleesMay}
{\sc J.~P.~C. Greenlees and J.~P. May}, {\em Equivariant stable homotopy
  theory}, in Handbook of algebraic topology, North-Holland, Amsterdam, 1995,
  pp.~277--323.

\bibitem{GuillouYarnall}
{\sc B.~Guillou and C.~Yarnall}, {\em The {K}lein four slices of
  {$\Sigma^nH\underline{\Bbb F}_2$}}, Math. Z., 295 (2020), pp.~1405--1441.

\bibitem{Haz21}
{\sc C.~Hazel}, {\em The {$RO(C_2)$}-graded cohomology of {$C_2$}-surfaces in
  {$\Bbb Z/2$}-coefficients}, Math. Z., 297 (2021), pp.~961--996.

\bibitem{Hill-sliceprimer}
{\sc M.~A. Hill}, {\em The equivariant slice filtration: a primer}, Homology
  Homotopy Appl., 14 (2012), pp.~143--166.

\bibitem{HHR}
{\sc M.~A. Hill, M.~J. Hopkins, and D.~C. Ravenel}, {\em On the nonexistence of
  elements of {K}ervaire invariant one}, Ann. of Math. (2), 184 (2016),
  pp.~1--262.

\bibitem{HHR-slicess}
{\sc M.~A. Hill, M.~J. Hopkins, and D.~C. Ravenel}, {\em The slice spectral
  sequence for certain {$RO(C_{p^n})$}-graded suspensions of {$H\underline{\bf
  Z}$}}, Bol. Soc. Mat. Mex. (3), 23 (2017), pp.~289--317.

\bibitem{HHR-realKthy}
{\sc M.~A. Hill, M.~J. Hopkins, and D.~C. Ravenel}, {\em The slice spectral
  sequence for the {$C_4$} analog of real {$K$}-theory}, Forum Math., 29
  (2017), pp.~383--447.

\bibitem{HillShiWangXu}
{\sc M.~A. Hill, X.~D. Shi, G.~Wang, and Z.~Xu}, {\em The slice spectral
  sequence of a {$C_4$}-equivariant height-4 {L}ubin-{T}ate theory}, 2018.
\newblock available at \url{https://arxiv.org/abs/1811.07960}.

\bibitem{HillYarnall}
{\sc M.~A. Hill and C.~Yarnall}, {\em A new formulation of the equivariant
  slice filtration with applications to {$C_p$}-slices}, Proc. Amer. Math.
  Soc., 146 (2018), pp.~3605--3614.

\bibitem{Hog21}
{\sc E.~Hogle}, {\em {$RO(C_2)$}-graded cohomology of equivariant
  {G}rassmannian manifolds}, New York J. Math., 27 (2021), pp.~53--98.

\bibitem{HoKr17}
{\sc J.~Holler and I.~Kriz}, {\em On {${\rm RO}(G)$}-graded equivariant
  ``ordinary'' cohomology where {$G$} is a power of {${\Bbb Z}/2$}}, Algebr.
  Geom. Topol., 17 (2017), pp.~741--763.

\bibitem{Lewis}
{\sc L.~G. Lewis, Jr.}, {\em The {$R{\rm O}(G)$}-graded equivariant ordinary
  cohomology of complex projective spaces with linear {${\bf Z}/p$} actions},
  in Algebraic topology and transformation groups ({G}\"{o}ttingen, 1987),
  vol.~1361 of Lecture Notes in Math., Springer, Berlin, 1988, pp.~53--122.

\bibitem{Lew92}
\leavevmode\vrule height 2pt depth -1.6pt width 23pt, {\em The equivariant
  {H}urewicz map}, Trans. Amer. Math. Soc., 329 (1992), pp.~433--472.

\bibitem{MandellMay}
{\sc M.~A. Mandell and J.~P. May}, {\em Equivariant orthogonal spectra and
  {$S$}-modules}, Mem. Amer. Math. Soc., 159 (2002), pp.~x+108.

\bibitem{May20}
{\sc C.~May}, {\em A structure theorem for {$RO(C_2)$}-graded {B}redon
  cohomology}, Algebr. Geom. Topol., 20 (2020), pp.~1691--1728.

\bibitem{May-EquivBook}
{\sc J.~P. May}, {\em Equivariant homotopy and cohomology theory}, vol.~91 of
  CBMS Regional Conference Series in Mathematics, Published for the Conference
  Board of the Mathematical Sciences, Washington, DC; by the American
  Mathematical Society, Providence, RI, 1996.
\newblock With contributions by M. Cole, G. Comeza\~{n}a, S. Costenoble, A. D.
  Elmendorf, J. P. C. Greenlees, L. G. Lewis, Jr., R. J. Piacenza, G.
  Triantafillou, and S. Waner.

\bibitem{MoTa68}
{\sc R.~E. Mosher and M.~C. Tangora}, {\em Cohomology operations and
  applications in homotopy theory}, Harper \& Row, Publishers, New York-London,
  1968.

\bibitem{Slone}
{\sc C.~Slone}, {\em Klein four 2-slices and the slices of {$\Sigma^{\pm
  n}H\underline{\Bbb Z}$}}, Math. Z., 301 (2022), pp.~3895--3938.

\bibitem{TW95}
{\sc J.~Th\'{e}venaz and P.~Webb}, {\em The structure of {M}ackey functors},
  Trans. Amer. Math. Soc., 347 (1995), pp.~1865--1961.

\bibitem{Ullman-slice}
{\sc J.~Ullman}, {\em On the slice spectral sequence}, Algebr. Geom. Topol., 13
  (2013), pp.~1743--1755.

\bibitem{Ull13}
\leavevmode\vrule height 2pt depth -1.6pt width 23pt, {\em Tambara functors and
  commutative ring spectra}, 2013.
\newblock available at \url{https://arxiv.org/pdf/1304.4912.pdf}.

\bibitem{Voe03}
{\sc V.~Voevodsky}, {\em Reduced power operations in motivic cohomology}, Publ.
  Math. Inst. Hautes \'{E}tudes Sci.,  (2003), pp.~1--57.

\bibitem{Was69}
{\sc A.~G. Wasserman}, {\em Equivariant differential topology}, Topology, 8
  (1969), pp.~127--150.

\bibitem{Yarnall-slice}
{\sc C.~Yarnall}, {\em The slices of {$S^n\wedge H\underline{\Bbb Z}$} for
  cyclic {$p$}-groups}, Homology Homotopy Appl., 19 (2017), pp.~1--22.

\bibitem{Yos83}
{\sc T.~Yoshida}, {\em On {$G$}-functors. {II}. {H}ecke operators and
  {$G$}-functors}, J. Math. Soc. Japan, 35 (1983), pp.~179--190.

\bibitem{Zeng}
{\sc M.~Zeng}, {\em Equivariant {E}ilenberg-{M}ac{L}ane spectra in cyclic
  $p$-groups}, 2018.

\end{thebibliography}
\end{document}